\documentclass[10pt]{article}  
\usepackage{mystyle}

\newtheorem{proposition}{Proposition}
\newtheorem{definition}{Definition}[section]

\DeclareMathOperator*{\argmax}{argmax}
\DeclareMathOperator*{\argmin}{argmin}

\begin{document}

\begin{center} 
{\Large Physics-Informed Tolerance Allocation: A Surrogate-Based Framework for the Control of Geometric Variation on System Performance}\\
\vspace{.2in}
J. Benzaken, A. Doostan, J. A. Evans
\end{center}

\section{Abstract}
In this paper, we present a novel tolerance allocation algorithm for the assessment and control of geometric variation on system performance that is applicable to any system of partial differential equations. In particular, we parameterize the geometric domain of the system in terms of design parameters and subsequently measure the effect of design parameter variation on system performance. A surrogate model via a tensor representation is constructed to map the design parameter variation to the system performance.  A set of optimization problems over this surrogate model restricted to nested hyperrectangles represents the effect of prescribing design tolerances, where the maximizer of this restricted function depicts the worst-case member, i.e. the worst-case design. Moreover, the loci of these tolerance hyperrectangles with maximizers attaining, but not surpassing, the performance constraint represents the boundary to the feasible region of allocatable tolerances. Every tolerance in this domain is measured through a user-specified, weighted norm which is informed by design considerations such as cost and manufacturability. The boundary of the feasible set is elucidated as an immersed manifold of codimension one, over which a suite of optimization routines exist and are employed to efficiently determine an optimal feasible tolerance with respect to the specified measure. Examples of this algorithm are presented with applications to a plate with a hole described by two design parameters, a plate with a hole described by six design parameters, and an L-Bracket described by seventeen design parameters.

\section{Introduction}

The ultimate goal of Computer-Aided Engineering (CAE) and Computer-Aided Design (CAD) is to design, manufacture, and maintain a product which performs some intended task within specified performance constraints. Throughout the engineering design cycle and the service lifetime of the assembled product, it is expected that uncertainty due to manufacturing processes, fatigue, modeling assumptions, etc. will affect the overall system performance. Regardless of these uncertainties, the product is still expected to successfully perform its intended task. 

As a particularly relevant example, geometric variations which occur during manufacturing adversely affect part performance. However, these effects on performance are rarely, if ever, considered during part design in a rigorous manner. Instead, manufacturing tolerances are often prescribed as to only address the issue of goodness-of-fit between parts in an assembly \cite{chase1988design,chase1990least,choi2000optimal,ngoi1999optimum}. A number of approaches are utilized to find optimal tolerances for this purpose, including linear and nonlinear programming \cite{chase1990least}, integer programming \cite{feng1997robust}, genetic algorithms \cite{ji2000optimal}, simulated annealing \cite{zhang1993integrated}, and ant colony optimization \cite{ramesh2009concurrent}, and Monte Carlo methods are typically employed in efforts to ascertain the certainty of fit in a statistical sense \cite{heling2016connected,lin1997study,wu2009improved}. Parts which are not in compliance with these tolerances are then recycled or discarded, imposing additional costs on manufacturing processes. Wear and tear throughout the part life span further deteriorate part performance. These issues are often addressed in an \textit{a priori} manner via safety factors or in an \textit{a posteriori} manner using measurements periodically collected throughout the part life cycle to assess the overall condition of the system. In this manner, the status of the engineering system is not well-understood until thorough subsequent analyses are conducted which, if not performed in a timely manner, may result in catastrophic failure.

In this paper, we present a methodology for determining spaces of admissible design parameters, e.g. tolerances, \textit{a priori}, given acceptable performance metrics. This allows the engineer to know, simply through measurement, the compliance of the system with regards to performance constraints. Additionally, this methodology allows for the implementation of systems which can detect their own non-compliance. This is accomplished by parameterizing the solution to partial differential equations, and relevant quantities of interest therein, as a function of geometric design parameters. We leverage isogeometric analysis \cite{hughes2005isogeometric} and a flavor of the surrogate modeling methodology presented in \cite{benzaken2017rapid} to accomplish this.

Thereafter, Monte Carlo samples are taken throughout this parametric domain about some nominal design and subsequently, a surrogate model to the solution is constructed using a low-rank, separated representation. A set of domain restrictions over this surrogate model effectively emulates tolerances by permitting geometric deviations within some hyperrectangle about the nominal design. An optimization problem posed over the restricted surrogate yields the worst offender of the elements contained within a given hyperrectangle. Moreover, the largest hyperrectangle such that the worst offender remains within the prescribed performance constraint is associated with the optimal tolerance to allocate, which naturally has a strong dependence on the norm used to measure the size of the hyperrectangle.  As such, we refer to the above optimization problem as the optimal tolerance allocation problem.

The optimal tolerance allocation problem is interpreted herein as a manifold optimization. This manifold is of dimension one less than the number of design parameters used in the geometric parameterization. Moreover, it represents the loci of hyperrectangles whose worst offender is equivalent to the performance constraint that is considered. Provided a weighted norm that is informed by considerations such as design sensitivities or manufacturing costs, the point with the largest value in this measure is considered optimal. In this paper, we employ manifold gradient ascent and manifold conjugate gradient optimization routines to determine this optimal value due to their favorable computational cost as compared with non-gradient optimization methods, though other state-of-the-art optimization algorithms may certainly be used.

An outline of the paper is as follows.  In Section \ref{sec:parametric_pdes}, we discuss our approach for quantifying the impact of geometric variation on system response through the numerical solution of parametric partial differential equations.  In Section \ref{sec:opt_tol_problem}, we present the optimal tolerance allocation problem and our proposed solution strategy.  In Section \ref{sec:NumTests}, we examine the effectiveness of the proposed solution strategy for the optimal tolerance allocation problem using three application problems: a plate with a hole described by two design parameters, a plate with a hole described by six design parameters, and an L-Bracket described by seventeen design parameters.  Finally, in Section \ref{sec:conclusions}, we draw conclusions.

\section{Propagation of Geometric Variation}
\label{sec:parametric_pdes}

In computational mechanics, we typically seek a finite-dimensional approximation to a solution of a system of partial differential equations (PDEs). For simplicity, consider a general boundary value problem (BVP) of the form:
\begin{equation}
\begin{array}{rl}
\bm{\mathcal{L}} \left(\vec{u}\right) &= \bm{\mathcal{F}} \ \ \ \forall \ {\bf x} \in \Omega\\
\bm{\mathcal{B}}\left(\vec{u}\right) &= \bm{\mathcal{G}} \ \ \ \forall \ {\bf x} \in \Gamma,
\label{eqn:generalPDE}
\end{array}
\end{equation}
where $\bm{\mathcal{L}}(\cdot)$ is a differential operator, possibly nonlinear, $\bm{\mathcal{B}}(\cdot)$ is a boundary operator, $\Omega$ is the physical domain, and $\Gamma$ is the boundary of the physical domain.  Both the differential and boundary operators act on the unknown solution field $\vec{u} : \Omega \rightarrow \mathbb{R}^d$.  In the context of structural mechanics, $\vec{u}$ typically denotes the displacement field and the dimension $d = 2$ or $d = 3$.

We choose to employ isogeometric analysis (IGA) as the discretization procedure throughout this paper, due to its natural parametric modeling framework \cite{benzaken2017rapid,herrema2017framework,hsu2015interactive}. However, the methodology presented herein is amenable to any numerical discretization choice and solution procedures, provided there exists an explicit parametrization between design parameters and analysis quantities of interest. In IGA, we use the same ${n_{dof}}$ basis functions for design and analysis: 
\begin{equation}
\text{Geometry:} \hspace{10pt} {\bf x}(\bm{\xi}) = \sum_{i=1}^{n_{dof}} \vec{P}_i \hat{N}_i (\bm{\xi}); \hspace{50pt} \text{Displacement:} \hspace{10pt} \vec{u}^h({\bf x}) = \sum_{i=1}^{n_{dof}} \vec{d}_i N_i({\bf x}),
\label{eqn:IGAdisp}
\end{equation}
where $\hat{N}_i(\bm{\xi})$ and $N_i({\bf x})$ are the Non-Uniform Rational B-Spline (NURBS) basis functions defined over the parametric and spatial domains, respectively, and the terms $\vec{P}_i, \vec{d}_i \in \mathbb{R}^{d \times n_{dof}}$ are referred to as \emph{control points} and \emph{control variables}, respectively.  The NURBS basis functions in physical space are defined as a push forward of the NURBS basis functions in parametric space, i.e. $N_i({\bf x}) = \hat{N}_i({\bf x}^{-1}({\bf x}))$, in accordance with the \textit{isoparametric concept}. With admissible trial functions defined, we invoke Galerkin's method by multiplying the PDE system by a corresponding NURBS weighting function and integrating by parts.  This yields a residual system of the form:
\begin{equation}
 \textbf{R} \left({\bf d}\right) = {\bf 0},
\end{equation}
where $\textbf{R}$ is a vector of residuals and $\textbf{d} = \left[ \vec{d}^T_1 \ \vec{d}^T_2 \ \ldots \ \vec{d}^T_{n_{dof}} \right]^T$ is a solution vector of displacement coefficients or control variables. In the linear setting, this equation reduces further to the linear system:
\begin{equation}
\textbf{K} \textbf{d} = \textbf{F},
\label{eqn:LinSolve}
\end{equation}
where $\textbf{K}$ is the system stiffness matrix and $\textbf{F}$ is the system forcing vector.  To assemble and solve this system, finite elements are constructed through a process known as \emph{B\'ezier extraction}, where a transformation, referred to as the extraction operator, is constructed that describes the B-spline basis locally in terms of the Bernstein polynomials \cite{borden2011isogeometric,scott2011isogeometric}. These elements are then assembled in a global stiffness matrix and a global system solve is performed to obtain the displacement vector. For an in-depth overview of IGA and its implementation, the reader is referred to \cite{cottrell2009isogeometric}.

We can explore the impact of changing design parameters on the resulting solution using parametric PDEs. Specifically, we may consider problem parameters such as material properties as well as geometric configurations. Formalizing this concept, let us construct a \emph{design space} $\mathcal{D} \subset \R^{d_\mu}$, where $d_\mu$ is the dimension of the parametric space associated with design parameters.  We refer to each member of $\mathcal{D}$, denoted as $\bm{\mu} \in \mathcal{D}$, as a \emph{design variable}, and it contains a selection of design parameters governing the material and geometric properties for a given design.  We assume throughout the paper that $\mathcal{D}$ is a hyperrectangle, that is, a Cartesian product of intervals: $\mathcal{D} = [a_1, b_1] \otimes [a_2,b_2] \otimes \ldots \otimes [a_{d_\mu},b_{d_\mu}]$. With this machinery, we are capable of parameterizing the PDE system \eqref{eqn:generalPDE} in terms of the design variable $\bm{\mu}$:
\begin{equation}
\begin{array}{rl}
\bm{\mathcal{L}} \left(\vec{u}(\bm{\mu});\bm{\mu}\right) &= \bm{\mathcal{F}}\left(\bm{\mu}\right) \ \ \ \forall \ {\bf x} \in \Omega_{\bm{\mu}}\\
\bm{\mathcal{B}} \left(\vec{u}(\bm{\mu});\bm{\mu}\right) &= \bm{\mathcal{G}}\left(\bm{\mu}\right) \ \ \ \forall \ {\bf x} \in \partial \Gamma_{\bm{\mu}},
\label{eqn:generalParaPDE}
\end{array}
\end{equation}
where $\Omega_{\bm{\mu}}$ and $\Gamma_{\bm{\mu}}$ denote the physical domain and boundary, respectively, that are parametrized by $\bm{\mu}$. In the isogeometric setting, the geometric description and discrete solution to such a problem can be expressed analogously to Eq. \eqref{eqn:IGAdisp} as:
\begin{equation}
\text{Geometry:} \hspace{10pt} {\bf x}(\bm{\xi}, \bm{\mu}) = \sum_{i=1}^{n_{dof}} \vec{P}_i (\bm{\mu}) \hat{N}_i (\bm{\xi}); \hspace{50pt} \text{Displacement:} \hspace{10pt} \vec{u}^h({\bf x}, \bm{\mu}) = \sum_{i=1}^{n_{dof}} \vec{d}_i (\bm{\mu}) N_i({\bf x}).
\label{eqn:parametricSolution}
\end{equation}
The effect of $\bm{\mu}$ on the physical basis functions is known through the isoparametric concept, that is, the bijective mapping between the parametric and physical domains. Therefore, the $\bm{\mu}$-dependancy in the solution field is rendered solely a function of the control variables $\vec{d}_i(\bm{\mu})$ which are determined by an analogous set of nonlinear algebraic equations:
\begin{equation}
\textbf{R} \left( \bm{\mu}, {\bf d}(\bm{\mu})  \right) = {\bf 0} \hspace{15pt} \forall \ \bm{\mu} \in \mathcal{D}
\label{eqn:nonlinearRes}
\end{equation}
where ${\bf R}(\bm{\mu}, {\bf d}(\bm{\mu}))$ is a vector of residuals and ${\bf d}(\bm{\mu})$ is a solution vector collecting the unknown control variables. Once again, in the linear setting, the solution of \eqref{eqn:nonlinearRes} reduces to solving the system:
\begin{equation}
\textbf{K}(\bm{\mu}) \ {\bf d}(\bm{\mu}) = \textbf{F}(\bm{\mu}) \hspace{15pt} \forall \ \bm{\mu} \in \mathcal{D}.
\label{eqn:parametricPDE}
\end{equation}
The discretized PDE systems \eqref{eqn:nonlinearRes} and \eqref{eqn:parametricPDE} provide a vessel to explore the propagation of geometric variations through the PDE. Namely, for every $\bm{\mu}$, the solution ${\bf d}(\bm{\mu})$ is the full parametric system response. Moreover, we will utilize the parameterizations presented herein to ultimately devise a methodology to assess system performance as a function of design variable, of which we begin discussion in the following section. 

\section{Tolerance Allocation}
\label{sec:opt_tol_problem}

The focus of this paper is on solving the \textit{tolerance allocation problem}, i.e. the allocation of tolerances such that some measure of performance is satisfied \cite{chase1999tolerance}.  Our solution to the tolerance allocation problem emanates naturally from the isogeometric design space exploration methodology presented in \cite{benzaken2017rapid}, and for the sake of consistency, we employ a similar notation throughout this paper. Therein, the authors presented a framework for obtaining surrogate models to the solution vector of the discretized PDE and hence the full system response as a function of the design variable, $\bm{\mu}$. In contrast to the design space exploration framework, rather than a full-system response, we are instead interested in the \emph{system performance} $\mathcal{Q}(\bm{\mu})$, a scalar-valued function which provides a quantity of interest (e.g. maximum stress) as a function of the design variable. 

We are particularly interested in controlling allowable geometric deviations such that any realized design conforms to a prescribed performance constraint, which we denote $\mathcal{Q}_\text{allow}$.  In this direction, we may associate a \emph{tolerance variable} $\bm{\tau} \in \R^{d_\mu}_+$ to a \emph{nominal design}, denoted by the design variable $\hat{\bm{\mu}} \in \R^{d_\mu}$.  We assume that $\mathcal{Q}(\hat{\bm{\mu}}) < \mathcal{Q}_\text{allow}$, that is, the nominal design does not violate the prescribed performance constraint.  A second design, denoted by $\tilde{\bm{\mu}} \in \R^{d_\mu}$, is said to be within the prescribed tolerance of the nominal design if $\left| \hat{\mu}_i - \tilde{\mu}_i \right| \le \tau_i$ for $i = 1,2,\ldots, d_\mu$. If all designs within the prescribed tolerance of the nominal design satisfy $\mathcal{Q}(\tilde{\bm{\mu}}) \le \mathcal{Q}_\text{allow}$, then they all conform to the prescribed performance constraint.  Unfortunately, while this point of view elucidates whether a particular tolerance variable is allowable or not, it does not yield insight as to how to define nor how to find an \textit{optimal tolerance variable} among the full space of tolerance variables.  To do so, we need to introduce a few additional concepts.  First, for each tolerance variable $\bm{\tau} \in \R^{d_\mu}_+$, let us associate a \emph{tolerance hyperrectangle}:
\begin{align}
\mathcal{D}_{\hat{\bm{\mu}}}(\bm{\tau}) := \left\{ \bm{\mu} \in \R^{d_\mu} : \left| \mu_i - \hat{\mu}_i \right| \le \tau_i , \ i = 1,2,\ldots,d_\mu \right\}.
\label{eqn:tolHypRect}
\end{align}
This hyperrectangle formally characterizes the space of designs which deviate from a nominal design $\hat{\bm{\mu}}$ within the threshold specified by $\bm{\tau}$. Note that sequences of hyperrectangles of this type are in fact nested. In particular, given some $\hat{\bm{\mu}}$ and tolerance variables $\bm{\tau}_1,\bm{\tau}_2$ with $(\bm{\tau}_1)_i \le (\bm{\tau}_2)_i$ for $i = 1,2,\ldots, d_\mu$, it necessarily follows that $\mathcal{D}_{\hat{\bm{\mu}}}(\bm{\tau}_1) \subseteq \mathcal{D}_{\hat{\bm{\mu}}}(\bm{\tau}_2)$.  Next, for each tolerance variable $\bm{\tau} \in \R^{d_\mu}_+$, let us associate a \textit{performance measure}:
\begin{align}
\mathcal{G}(\bm{\tau}) := \max_{\bm{\mu} \in \mathcal{D}_{\hat{\bm{\mu}}}(\bm{\tau})} \mathcal{Q}(\bm{\mu}).
\label{eqn:tolConstraint}
\end{align}
While the quantity $\mathcal{Q}(\hat{\bm{\mu}})$ characterizes the system performance of the nominal design, the quantity $\mathcal{G}(\bm{\tau})$ alternatively characterizes the worst-case system performance among all designs within the prescribed tolerance of the nominal design.  If it holds that $\mathcal{G}(\bm{\tau}) \leq \mathcal{Q}_\text{allow}$, then the tolerance variable $\bm{\tau}$ is allowable.  Note that the performance measure is monotonic in each of its components.  That is, $\mathcal{G}(\bm{\tau}_1) \leq \mathcal{G}(\bm{\tau}_2)$ for any two tolerance variables $\bm{\tau}_1,\bm{\tau}_2$ with $(\bm{\tau}_1)_i \le (\bm{\tau}_2)_i$ for $i = 1,2,\ldots, d_\mu$.  Finally, let $\mathcal{F}(\bm{\tau})$ denote a prescribed \textit{tolerance measure} which returns a measure of the size of a tolerance variable.  We assume throughout this paper that the tolerance measure is also monotonic in each of its components.  We will later review potential choices for the tolerance measure $\mathcal{F}(\bm{\tau})$.  With all the above concepts established, we are finally ready to define the optimal tolerance variable.  Namely, the optimal tolerance variable is that which maximizes the tolerance measure over all acceptable tolerance variables.  Mathematically, it is the solution of the following optimization problem:
\begin{framed}
Given $\hat{\bm{\mu}} \in \R^{d_\mu}_+$, find $\hat{\bm{\tau}}$ such that
\begin{align}
\hat{\bm{\tau}} = \argmax_{\bm{\tau} \in \mathcal{T}_\text{allow}} \mathcal{F}(\bm{\tau}), \hspace{15pt} \text{where} \hspace{15pt} \mathcal{T}_\text{allow} := \left\{ \bm{\tau} \in \R^{d_\mu}_+ : \mathcal{G}(\bm{\tau}) \le \mathcal{Q}_\text{allow}\right\}.
\label{eqn:OptProb}
\end{align}
\end{framed}

We refer to the above as the \textit{optimal tolerance allocation problem}.  It should be noted the above is a \textit{worst-case tolerance allocation problem} as the corresponding performance measure is associated with worst-case system performance.  One may alternatively consider a \textit{statistical tolerance allocation problem} wherein a statistical performance measure is employed.  The framework presented herein may be easily extended to statistical tolerance allocation, though other surrogate modeling methodologies, such as those based on polynomial chaos expansions, may be better suited for such a setting.  It should also be noted that the problem of optimal tolerance allocation is closely related to the problems of \textit{robust optimal design}, wherein the nominal design itself is optimized in the presence of uncertainty \cite{parkinson1993general}, and \textit{process capability optimization}, wherein both the means and variances of process specifications are optimized \cite{flaig2002process}.

While the optimal tolerance allocation problem given by \eqref{eqn:OptProb} succinctly characterizes an optimal tolerance variable, there are three outstanding concerns associated with both its definition and solution:

\begin{enumerate}
\item[1.]{Since each function call to $\mathcal{G}(\bm{\tau})$ is a demanding optimization over $\bm{\mu}$, how can one mitigate this seemingly unavoidable computational expense?}
\item[2.]{Since the choice of $\mathcal{F}(\bm{\tau})$ may dramatically affect the resulting $\hat{\bm{\tau}}$, what is an appropriate choice for this measure?}
\item[3.]{How can one arrive at $\hat{\bm{\tau}}$, the solution to the optimization problem Eq. \eqref{eqn:OptProb} which we have outlined above?}
\end{enumerate}

We address these issues in the following subsections and present our solutions to each.

\subsection{Low-Rank, Separated Representation of System Performance}
\label{section:SRConstruct}

Given a tolerance variable $\bm{\tau}$, the performance measure $\mathcal{G}(\bm{\tau})$ searches the corresponding tolerance hyperrectangle for the entry which maximizes the system performance $\mathcal{Q}(\bm{\mu})$, a rather costly optimization procedure. This is due to the necessity of a global system construction and subsequent solve of the isogeometric discretization for each $\bm{\mu}$. Moreover, this computational cost is only compounded in optimization routines over the tolerance variable, in which $\mathcal{G}(\bm{\tau})$ must be evaluated several times.  The goal of this subsection is to construct an economical and numerically stable model for the quantity $\mathcal{Q}(\bm{\mu})$. 

To alleviate the inherent computational expense of computing $\mathcal{Q}(\bm{\mu})$, we resort to constructing a surrogate model, analogous to those considered in \cite{benzaken2017rapid}. However, the nodal and modal surrogate models considered therein suffer from the notorious curse of dimensionality, where a linear increase in model fidelity demands an exponential increase in required sample realizations. Additionally, high-fidelity orthogonal polynomial expansions are comprised of many terms, consequently increasing computational expense for such constructions. Therefore, we instead adopt a technique emanating from the tensor approximation community known as \emph{low-rank, separated representation} \cite{beylkin2002numerical,beylkin2005algorithms}. That is, we employ a representation of the form:
\begin{align}
 \mathcal{Q}(\bm{\mu}) \approx \tilde{\mathcal{Q}}_{r}(\bm{\mu}) = \sum_{\ell = 1}^r s_\ell \mathscr{G}_{\ell}(\bm{\mu}) \hspace{15pt} \text{where} \hspace{15pt} \mathscr{G}_{\ell}(\bm{\mu}) = \prod_{i=1}^{d_\mu} g_{\ell}^i(\mu_i)
 \label{eqn:SR}
\end{align}
for a surrogate model to $\mathcal{Q}(\bm{\mu})$. The separation rank, $r$, is chosen to be relatively small, hence the term \textit{low-rank}, mitigating the stability and economic concerns presented above.  The coefficients $s_\ell$ are constants which enforce any normalization preferences, e.g. $\|g_{\ell}^i\| = 1$.  Unlike in standard approximation where the basis functions are chosen \textit{a priori}, in separated representation, the basis functions (or factors) $\mathscr{G}_{\ell}(\bm{\mu})$ and the separation rank $r$ are computed such that $\tilde{\mathcal{Q}}_r(\bm{\mu})$ is as close as possible to $\mathcal{Q}(\bm{\mu})$ (in a sense to be specified below).  This, along with the tensor-product construction of the factors $\mathscr{G}_{\ell}(\bm{\mu})$ in \eqref{eqn:SR}, renders the construction of separated representation as a non-linear optimization problem with various solution approaches as outlined in \cite{beylkin2009multivariate, doostan2013non}.  A discrete formulation of this optimization problem can be obtained by approximating each univariate factor $g_{\ell}^i(\mu_i)$ in, for instance, a polynomial basis $\{\mathscr{L}_j(\mu_i)\}_{j=0}^{p}$.  This yields the approximation:
\begin{align}
g_{\ell}^i(\mu_i) \approx \sum_{j=0}^p c_{\ell,j}^i \mathscr{L}_j(\mu_i),
\end{align}
and we denote the resulting surrogate model by $\tilde{\mathcal{Q}}_{r,p}(\bm{\mu})$.  In the present study we choose $\mathscr{L}_j$ to be, up to a shifting, the Legendre polynomial of degree $j$.  With this discretization, the problem of constructing a separated representation of $\mathcal{Q}(\bm{\mu})$ simplifies to that of computing the coefficients $c_{\ell,j}^i$ as well as the separation rank $r$.  To this end, one approach utilizes an alternating least-squares (ALS) routine, e.g. see \cite{bro1997parafac,de2000best,kroonenberg1980principal,sibileau2018explicit}, which minimizes the usual least-squares error between $N$ data points $\left\{ \left( \bm{\mu}^{(j)}, \mathcal{Q}^{(j)} \right) \right\}_{j=1}^N$ and the minimizer $\tilde{\mathcal{Q}}_{r,p}(\bm{\mu})$:
\begin{align}
\| \{ ( \bm{\mu}^{(j)}, \mathcal{Q}^{(j)} ) \}_{j=1}^N - \tilde{\mathcal{Q}}_{r,p}(\bm{\mu}^{(j)}) \|^2 = \sum_{j=1}^N (\mathcal{Q}^{(j)} - \tilde{\mathcal{Q}}_{r,p}(\bm{\mu}^{(j)}))^2
\label{eqn:LS}
\end{align}
where $\bm{\mu}^{(j)} = (\mu_1^{(j)},\mu_2^{(j)},\ldots,\mu_{d_\mu}^{(j)})$ and $\mathcal{Q}^{(j)} = \mathcal{Q}( \bm{\mu}^{(j)})$.  The ALS  scheme approximates the optimization of \eqref{eqn:LS} -- a nonlinear program -- via a sequence of linear optimization problems over each dimension $\mu_i$ one at a time. In particular, for a fixed $r$, ALS minimizes \eqref{eqn:LS} over $\{c_{\ell,j}^{i}\}$ while freezing the coefficients $\{c_{\ell,j}^{k}\}$, $k\ne i$, along other directions at their current values. It then alternates over all other directions to corresponding coefficients. This process is repeated multiple times until the residual in \eqref{eqn:LS} does not change much. In the case the achieved residual is larger than a prescribed tolerance, $r$ is increased and the ALS process is repeated. It should be noted several approaches have been recently proposed to improve the performance of ALS \cite{battaglino2018practical,doostan2013non,reynolds2016randomized}.  Moreover, other orthogonal polynomials, e.g. Hermite, may be employed to consequently induce a non-uniform weighting on the sample space \cite{xiu2002wiener}.

Low-rank, separated representations are particularly attractive for optimal tolerance allocation for two reasons.  First, the low-rank nature of separated representations affords them numerically stable and economical evaluation, since they are comprised of relatively few terms in comparison to potentially thousands in an orthogonal polynomial expansion.  Second, there exist highly efficient and stable algorithms for computing the maximum value reached by a low-rank, separated representation over a hyperrectangle \cite{reynolds2017optimization}.  This is highly convenient for optimal tolerance allocation since computing the performance measure $\mathcal{G}(\bm{\tau})$ for a particular tolerance variable $\bm{\tau}$ involves computing the maximum system performance $\mathcal{Q}(\bm{\mu})$ over the tolerance hyperrectangle $\mathcal{D}_{\hat{\bm{\mu}}}(\bm{\tau})$.

\subsection{Construction of a Sampling Domain for Surrogate Model Construction}
\label{section:hyperrectangleConstruct}

In order to construct a low-rank, separated representation of the system performance $\mathcal{Q}(\bm{\mu})$, one must evaluate $\mathcal{Q}(\bm{\mu})$ for a large sample of design variables.  Thus the problem of constructing a suitable \textit{sampling domain} $\mathcal{D}_{\text{sample}} \subseteq \mathcal{D}$ arises.  If the sampling domain is chosen to be either too large or too small, then the resulting surrogate model $\tilde{\mathcal{Q}}_{r,p}(\bm{\mu})$ will be a poor approximation of the true system performance $\mathcal{Q}(\bm{\mu})$ for the design variables of interest.  To arrive at a suitable sampling domain, we require that it at least enclose every tolerance hyperrectangle associated with a tolerance variable in the space of allowable tolerance variables, viz.:
\begin{align}
\left\{ \mathcal{D}_{\hat{\bm{\mu}}}(\bm{\tau}): \bm{\tau} \in \mathcal{T}_{\text{allow}} \right\} \subseteq \mathcal{D}_{\text{sample}}.
\end{align}
To meet this constraint without resorting to an excessively large domain, we choose the sampling domain to be the smallest hyperrectangle enclosing every tolerance hyperrectangle, viz.:
\begin{align}
\mathcal{D}_{\text{sample}} := \otimes_{i=1}^{d_\mu} \left[ \hat{\mu}_i - \left( \bm{\tau}_{\text{max}} \right)_i, \hat{\mu}_i + \left( \bm{\tau}_{\text{max}} \right)_i \right]
\end{align}
where:
\begin{align}
\left( \bm{\tau}_{\text{max}} \right)_i := \max \left\{ \tau_i : \bm{\tau} \in \mathcal{T}_{\text{allow}} \right\}.
\end{align}
Note that $\left( \bm{\tau}_{\text{max}} \right)_i$ is the largest possible size of the $i^{\text{th}}$ side of a tolerance hyperrectangle associated with a tolerance variable in the space of allowable tolerance variables.  Therefore, the following nestedness property holds:
\begin{align}
\mathcal{T}_{\text{allow}} \subseteq \otimes_{i=1}^{d_\mu} \left[ 0, \left( \bm{\tau}_{\text{max}} \right)_i \right].
\end{align}
We will later exploit the above property when constructing algorithms for finding the optimal tolerance variable.  In particular, we will restrict our search for optimal tolerances within the \textit{tolerance bounding box}:
\begin{align}
\mathcal{T}_{\text{bounding}} := \otimes_{i=1}^{d_\mu} \left[ \left( \bm{\tau}_{\text{min}} \right)_i, \left( \bm{\tau}_{\text{max}} \right)_i \right],
\end{align}
where $\bm{\tau}_\text{min}$ encodes potential user-specified lower bounds on the allowable tolerance variable.  Unless otherwise specified,  $\bm{\tau}_\text{min} = \bm{0}$.  Now, since the extent of the space of allowable tolerance variables is not known \textit{a priori}, neither is $\left( \bm{\tau}_{\text{max}} \right)_i$ for $i = 1, \ldots, d_\mu$.  Fortunately, since tolerance hyperrectangles are necessarily nested and the performance measure is monotonic in each of its components, it follows that:
\begin{align}
\left( \bm{\tau}_{\text{max}} \right)_i = \max \left\{ \tau_i : \bm{\tau} \in \mathcal{T}_{\text{allow}} \text{ and } \tau_j = 0 \text{ for } j \neq i \right\}.
\end{align}
Thus, $\left( \bm{\tau}_{\text{max}} \right)_i$ can be obtained by solving a univariate optimization problem over the space of allowable tolerances.  In fact, $\left( \bm{\tau}_{\text{max}} \right)_i$ can be found in terms of a univariate optimization problem over the space of design variables as illustrated by the following proposition.

\begin{framed}
\begin{proposition}
Define $q_i(\mu) := \mathcal{Q}(\hat{\bm{\mu}} + \mu{\bf e}_i)$ where ${\bf e}_i$ is the unit vector in the $i^{\text{th}}$ coordinate direction.  If $q_i(\mu) \in C^1(\mathbb{R})$, it holds that $\left( \bm{\tau}_{\text{max}} \right)_i = \left| \mu^*_i \right|$ where:
\begin{align}
\mu^*_i = \argmin_\mu \frac{1}{2} \left| \mu \right|^2 \hspace{15pt} \text{such that} \hspace{15pt} \left\{ \begin{array}{rl} q_i(\mu) &= \mathcal{Q}_\text{allow}\\
\frac{dq_i(\mu)}{d\mu} \cdot \sgn\left( \mu \right) &> 0 \end{array} \right.
\label{eqn:1Dopt}
\end{align}
for $i=1,2,\ldots,d_\mu$.
\end{proposition}

\begin{proof} Let $\tau^* = \left| \mu^*_i - \hat{\mu}_i \right|$ where:
\begin{align}
\mu^*_i = \argmin_\mu \frac{1}{2} \left| \mu - \hat{\mu}_i \right|^2 \hspace{15pt} \text{such that} \hspace{15pt} \left\{ \begin{array}{rl} q_i(\mu) &= \mathcal{Q}_\text{allow}\\
\frac{dq_i(\mu)}{d\mu} \cdot \sgn\left( \mu \right) &> 0. \end{array} \right.
\end{align}
It suffices to show two things: (i) $\bm{\tau}^*$ with $\tau^*_i = \tau^*$ and $\tau^*_j = 0$ for $j \neq i$ is a member of $\mathcal{T}_{\text{allow}}$ and (ii) any $\bm{\tau}$ with $\tau_i > \tau^*$ and $\tau_j = 0$ for $j \neq i$ is not a member of $\mathcal{T}_{\text{allow}}$.  We first prove (i) by showing that $q_i(\mu) \leq \mathcal{Q}_\text{allow}$ for all $|\mu| \leq \tau^*$.  Suppose instead that $q_i(\mu) > \mathcal{Q}_\text{allow}$ for some $|\mu| \leq \tau^*$.  Then there is some $|\tilde{\mu}| < \tau^*$ such that $q_i(\tilde{\mu}) = \mathcal{Q}_\text{allow}$ and $\frac{dq_i(\tilde{\mu})}{d\mu} \cdot \sgn\left( \tilde{\mu} \right) > 0$.  This violates the definition of $\mu^*_i$, so (i) holds. To prove (ii), note that $q_i(\mu^*_i+\epsilon \sgn\left( \mu \right)) > \mathcal{Q}_\text{allow}$ for any $\epsilon > 0$ since $q_i(\mu^*_i) = \mathcal{Q}_\text{allow}$ and $\frac{dq_i(\mu^*_i}{d\mu} \cdot \sgn\left( \mu^*_i \right) > 0$.  Thus any $\bm{\tau}$ with $\tau_i > \tau^*$ and $\tau_j = 0$ for $j \neq i$ is necessarily not a member of $\mathcal{T}_{\text{allow}}$.
\end{proof}
\end{framed}

As opposed to the optimal tolerance allocation problem, the univariate optimization problem given by \eqref{eqn:1Dopt} can be efficiently solved using Newton's method or a quasi-Newton method.  Consequently, $\bm{\tau}_{\text{max}}$ may be efficiently computed without resorting to a surrogate model for the system performance $\mathcal{Q}(\bm{\mu})$.  This is critical to the functionality of our methodology for solving the optimal tolerance allocation problem, since we construct the surrogate model $\tilde{\mathcal{Q}}_{r,p}(\bm{\mu})$ based on sampling the domain $\mathcal{D}_{\text{sample}}$ whose size is dictated by $\bm{\tau}_{\text{max}}$.

\subsection{Tolerance Measures}

Recall that the objective of the optimal tolerance allocation problem is to find the optimal tolerance variable $\hat{\bm{\tau}}$ in the space of allowable tolerance variables $\mathcal{T}_{\text{allow}}$ which maximizes the tolerance measure $\mathcal{F}(\bm{\tau})$.  Consequently, the precise specification of the tolerance measure has a profound impact on the solution of the optimal tolerance allocation problem.  Perhaps the simplest possible specification is a weighted sum of the individual components of the tolerance variable.  This leads to a tolerance measure that is easy to study, and the weights associated with this specification may be tied, for example, to anticipated cost or design parameter sensitivity.  Another means of specifying the tolerance measure is to directly relate it to manufacturing cost $C(\bm{\tau})$.  Namely, if one seeks to minimize manufacturing cost, the tolerance measure should vary inversely with the cost, i.e. $\mathcal{F}(\bm{\tau}) \sim 1/C(\bm{\tau})$.  It is common to assume that manufacturing cost can be expressed as a sum of costs, each associated to a single component of the tolerance variable, viz.:
\begin{align}
C(\bm{\tau}) := \sum_{i=1}^{d_{\bm{\mu}}} c_i(\tau_i),
\end{align}
and there exist several candidate models for the cost associated with a single component of the tolerance variable \cite{chase1999tolerance}.  Perhaps the most popular model is the reciprocal power model:
\begin{align}
c_i(\tau_i) := a_i + b_i/\tau_i^{k_i}
\end{align}
where $a_i$, $b_i$, and $k_i$ are model constants which may be empirically determined for a particular application \cite{chase1990least}.  When $k_i = 1$, one recovers the reciprocal model of Chase and Greenwood \cite{chase1988design}, and when $k_i = 2$, one recovers the reciprocal squared model of Spotts \cite{spotts1973allocation}.  One particularly nice feature of the reciprocal power model is that it states manufacturing cost increases exponentially fast as tolerances are tightened, in accordance with engineering intuition.

In this paper, we examine three different model tolerance measures for optimal tolerance allocation.  Our purpose in doing so is to study the impact of tolerance measure specification on the form of the optimal tolerance variable as well as the effectiveness of the proposed numerical methodology for optimal tolerance allocation. The tolerance measures presented here are by no means exhaustive, nor are they empirically tied to manufacturing cost.  However, they are sufficiently diverse as to demonstrate the versatility of our methodology, and one measure we consider is related to the reciprocal model of manufacturing cost.
\begin{enumerate}
\item[1.]{The first measure considered, $\mathcal{F}_1(\bm{\tau})$, is perhaps the simplest choice:

\begin{equation}
\begin{aligned}
\mathcal{F}_1(\bm{\tau}) &:= \sum_{i=1}^{d_\mu} \tau_i.
\end{aligned}
\label{eqn:tolNorm_1}
\end{equation}

Note that $\mathcal{F}_1(\bm{\tau}) \equiv \| \bm{\tau}\|_1$. Intuitively, the selection of this tolerance functional will maximize the total tolerance available while complying to the performance constraint. However, this choice may lead to \textit{sparse} tolerance variables, i.e. tolerance variables with many components set to zero, in the presence of large discrepancies in the magnitudes of design parameter sensitivities.  We refer to this measure as the 1-norm.
}

\item[2.]{The second measure considered, $\mathcal{F}_{\bm{\mu}}(\bm{\tau})$, is a weighted sum of the individual components of the tolerance variable:
\begin{equation}
\begin{aligned}
\mathcal{F}_{\bm{\mu}}(\bm{\tau}) &:= \sum_{i=1}^{d_\mu} \left| \frac{\partial \mathcal{Q}}{\partial \mu_i}(\hat{\bm{\mu}}) \right| \tau_i
\end{aligned}
\label{eqn:tolNorm_mu}
\end{equation}
The weights $\alpha_i = \left| \partial_{\mu_i} \mathcal{Q}(\hat{\bm{\mu}}) \right|$ account for design parameter sensitivities revealed through the surrogate model.   As such, design parameters which are the most sensitive to perturbations are given prevalence in the tolerance allocation.  We refer to this measure as the $\bm{\mu}$-norm.
}

\item[3.]{The third and final measure considered, $\mathcal{F}_{-1}(\bm{\tau})$, is derived from the reciprocal model of manufacturing cost.  In particular, selecting $a_i = 0$ and $b_i = 1$ in the reciprocal model yields:
\begin{equation}
\begin{aligned}
\mathcal{F}_{-1}(\bm{\tau}) &:= \left( \sum_{i=1}^{d_\mu} \frac{1}{\tau_i} \right)^{-1}.
\end{aligned}
\label{eqn:tolNorm_-1}
\end{equation}
The above measure tends to avoid the selection of sparse tolerance variables (since sparse tolerance variables result in infinite manufacturing cost), and it also tends to result in \textit{isotropic} tolerance variables, i.e. tolerance variables whose components are comparable in size.  We refer to this measure as the $-1$-norm.
}
\end{enumerate}

It should be noted that in order for the above model tolerance measures to be well-defined, each of the design parameters should be expressed in terms of the same physical units.  We consider dimensionless design parameters in our later numerical experiments.

\subsection{Optimal Tolerance Allocation by Manifold Traversal}

We are now ready to discuss our algorithm for solving the optimal tolerance allocation problem.  The algorithm exploits advances in the field of manifold optimization, and in particular, we find the optimal tolerance variable on the surface of the following manifold describing the boundary of the space of allowable tolerance variables:
\begin{align}
\mathcal{M} := \left\{ \bm{\tau} \in \R^{d_{\bm{\mu}}}_+ : \mathcal{G}(\bm{\tau}) = \mathcal{Q}_\text{allow} \right\}.
\label{eqn:ManifoldDef}
\end{align}
However, classical optimization algorithms which operate in Euclidean space cannot be directly employed to find the optimal tolerance variable on $\mathcal{M}$ unless there exists a global parameterization from Euclidean space to the manifold.  Fortunately, many optimization algorithms have been generalized to the manifold setting by introducing an affine connection between tangent spaces corresponding to different points along a manifold \cite{absil2009optimization}.  Common to these algorithms is that, for every iteration, a tangent space $T_{\bm{\tau}} \mathcal{M}$ is constructed about a point $\bm{\tau} \in \mathcal{M}$. See, for instance, Fig. \ref{fig:man_illus}.

If we define $\mathcal{N}(\bm{\tau}) := \text{span} \left\{ \nabla_{\bm{\tau}} \mathcal{G}(\bm{\tau}) \right\}$ for $\bm{\tau} \in \mathcal{M}$, then it follows that $T_{\bm{\tau}}\mathcal{M} = \mathcal{N}(\bm{\tau})^\perp$.  Thus to characterize the tangent space $T_{\bm{\tau}}\mathcal{M}$, we need to be able to compute $\nabla_{\bm{\tau}} \mathcal{G}(\bm{\tau})$ in an efficient and accurate manner. Since $\mathcal{G}(\bm{\tau})$ is defined through an optimization problem and potentially exhibits sharp gradients, employing a finite difference stencil to approximate the gradient $\nabla_{\bm{\tau}} \mathcal{G}(\bm{\tau})$ is inefficient, inaccurate, and generally unstable.  As such, we compute $\nabla_{\bm{\tau}} \mathcal{G}(\bm{\tau})$ in an analytical manner. To do so, we first find the set of candidate $\bm{\mu}$-maximizers which define $\mathcal{G}(\bm{\tau})$:
\begin{align}
\mathcal{K}(\bm{\tau}) := \left\{ \bm{\mu} \in \mathcal{D}_{\hat{\bm{\mu}}}(\bm{\tau}) : \mathcal{Q}(\bm{\mu}) = \mathcal{G}(\bm{\tau}) \right\}.
\end{align}
We next find the candidate $\bm{\mu}$-maximizers which lie on the boundary of the tolerance hyperrectangle $\mathcal{D}_{\hat{\bm{\mu}}}(\bm{\tau})$, and we place them in the sets:
\begin{align}
\mathcal{K}_{\Gamma_i}(\bm{\tau}) := \left\{ \bm{\mu} \in \mathcal{K}(\bm{\tau}) : \left| \mu_i - \hat{\mu}_i\right| = \tau_i \right\}.
\end{align}
The components of the gradient $\nabla_{\bm{\tau}} \mathcal{G}(\bm{\tau})$ are then given by:
\begin{align}
\frac{\partial \mathcal{G}}{\partial \tau_i} (\bm{\tau}) = \left\{ \begin{array}{rl} 
\displaystyle \max_{\bm{\mu} \in \mathcal{K}_{\Gamma_i}(\bm{\tau})}  \max \left\{ \frac{\partial \mathcal{Q}}{\partial \mu_i} (\bm{\mu}) \cdot \sgn(\mu_i - \hat{\mu}_i), 0 \right\} ,& \mathcal{K}_{\Gamma_i}(\bm{\tau}) \neq \emptyset \\
0 ,& \mathcal{K}_{\Gamma_i}(\bm{\tau}) = \emptyset. \end{array} \right.
\end{align}
From the above expression, we see that $\frac{\partial \mathcal{G}}{\partial \tau_i} (\bm{\tau})$ is non-negative for $i = 1, \ldots, d_\mu$.  This is due to the fact that $\mathcal{G}(\bm{\tau})$ is monotonic in each of its components.  

\begin{figure}[t!]
\includegraphics{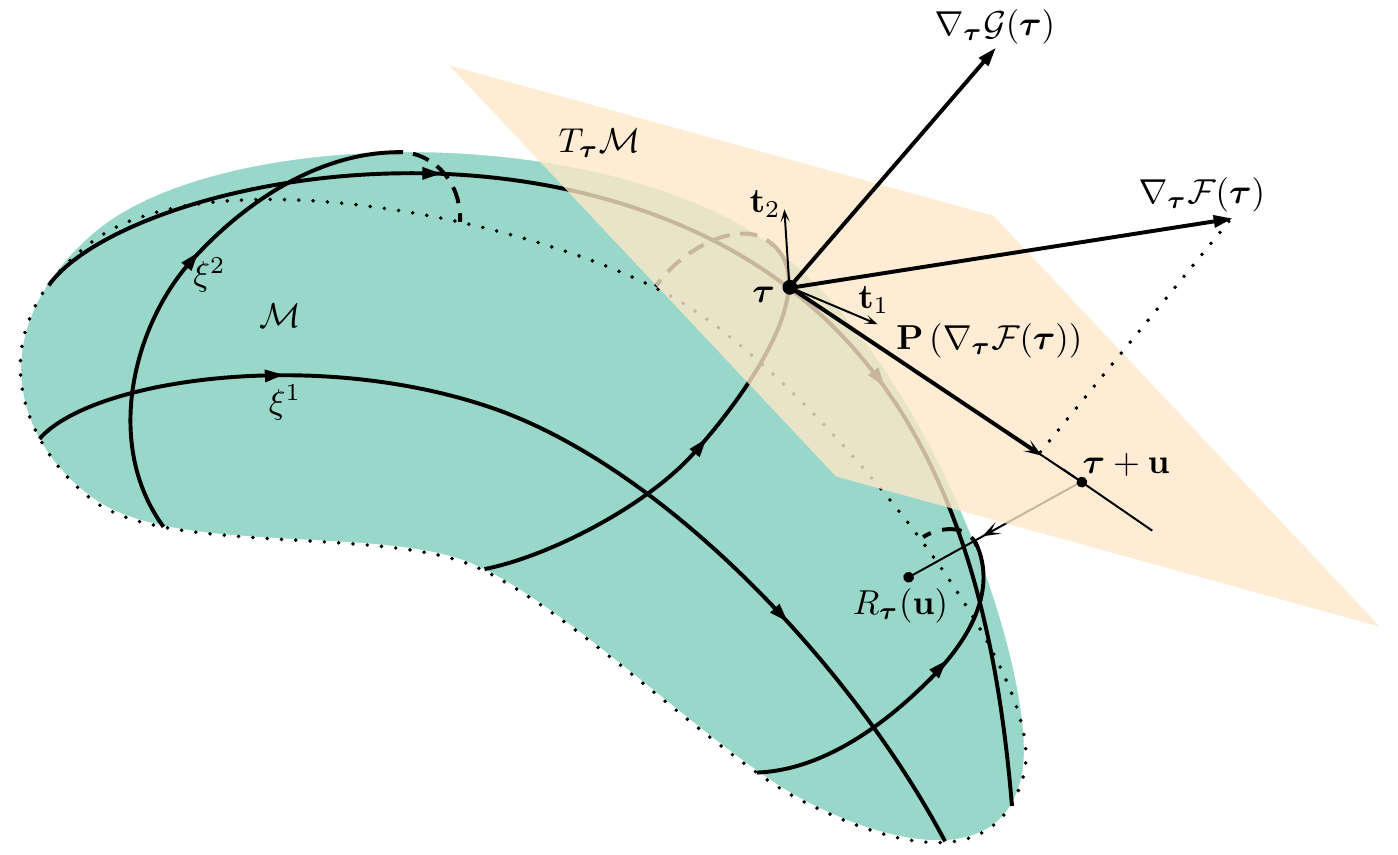}
\caption{An example manifold with a graphical illustration of the differential geometric tools presented in this section.}
\label{fig:man_illus}
\end{figure}

It necessarily holds that $\R^{d_{\bm{\mu}}} = \mathcal{N}(\bm{\tau}) \oplus T_{\bm{\tau}}\mathcal{M}$ for every $\bm{\tau} \in \mathcal{M}$. Let $\left\{ \bm{t}_i \right\}_{i=1}^{d_{\bm{\mu}} - 1} \in \R^{d_{\bm{\mu}}}$ denote an orthonormal basis of $T_{\bm{\tau}} \mathcal{M}$ and define ${\bf T} = \left[ {\bf t}_1 | {\bf t}_2 | \cdots | {\bf t}_{d_\mu-1} \right]$. Then for every $\bm{\tau} \in \R^{d_{\bm{\mu}}}$, there exists $b \in \R$ and $\bm{\xi} \in \R^{d_\mu-1}$ such that $\bm{\tau} = \bm{\eta} + b \hat{\bm{n}}$ where $\hat{\bm{n}} = \nabla_{\bm{\tau}} \mathcal{G}(\bm{\tau})/\| \nabla_{\bm{\tau}} \mathcal{G}(\bm{\tau}) \|$ and $\bm{\eta} = {\bf T} \bm{\xi}$. Intuitively, $\hat{\bf n}$ is the manifold normal at $\bm{\tau}$ and $\bm{\eta}$ is a vector in the tangent space with coordinates $\bm{\xi}$ in the ${\bf T}$ basis.  It is through the decomposition $\bm{\tau} = \bm{\eta} + b \hat{\bm{n}}$ that we are capable of performing manifold optimization solely in terms of tangent space entities. For example, with a gradient ascent method, we project $\nabla_{\bm{\tau}}\mathcal{F}(\bm{\tau})$ into the tangent space via the calculation ${\bf P} \left( \nabla_{\bm{\tau}}\mathcal{F}(\bm{\tau}) \right)$ where ${\bf P} = {\bf T} {\bf T}^T$ is the orthogonal projection matrix onto the tangent space.  This then provides the direction of steepest ascent.  An illustration of an example manifold $\mathcal{M}$, along with the necessary differential geometric tools to perform optimization, are shown in Fig. \ref{fig:man_illus}.

With a gradient ascent method, we progress along the direction of steepest ascent to improve the solution at each iteration.  However, any finite traversal along the tangent space $T_{\bm{\tau}} \mathcal{M}$ away from a point $\bm{\tau}$ on the manifold $\mathcal{M}$ will result in a departure from the manifold unless the manifold is flat.  Therefore, we require an operator which can take us back onto the manifold $\mathcal{M}$.  A so-called \emph{retraction} from the \textit{tangent bundle} $T\mathcal{M}: = \left\{ (\bm{\tau},{\bf{u}}) \in \mathbb{R}^{d_\mu} \times \mathbb{R}^{d_\mu}: {\bm{\tau}} \in \mathcal{M}, {\bf{u}} \in T_{\bm{\tau}} \mathcal{M} \right\}$ onto the manifold $\mathcal{M}$ is well-suited for this purpose. Retraction is formally defined below for our problem setting, and for further details, the reader is referred to \cite{smith1994optimization}.

\begin{framed}
\begin{definition}[Retraction]
A smooth mapping $R: T\mathcal{M} \rightarrow \mathcal{M}$ is said to be a \emph{retraction} if it satisfies the following properties where $R_{\bm{\tau}}$ is the restriction of $R$ to $\bm{\tau}$ (i.e. $R_{\bm{\tau}}({\bf{u}}) := R(\bm{\tau},{\bf{u}})$):
\begin{enumerate}
\item[(i)]{$R_{\bm{\tau}}(\bm{0}) = \bm{\tau}$ for all $\bm{\tau} \in \mathcal{M}$.}
\item[(ii)]{$DR_{\bm{\tau}}(\bm{0})[\bm{0}] = \text{id}_{T_{\bm{\tau}} \mathcal{M}}$ for all $\bm{\tau} \in \mathcal{M}$ where $D$ denotes the Fr\'{e}chet derivative operator and $\text{id}_{T_{\bm{\tau}} \mathcal{M}}$ is the identity mapping on $T_{\bm{\tau}} \mathcal{M}$.}
\end{enumerate}
\label{def:retraction}
\end{definition}
\end{framed}

The mathematically ideal retraction is the \textit{Riemannian exponential map} which maps a point $\bm{\tau} \in \mathcal{M}$ and tangent vector ${\bf{u}} \in T_{\bm{\tau}} \mathcal{M}$ to a point along a \textit{geodesic curve} on the manifold $\mathcal{M}$ which starts at $\bm{\tau}$ in the direction of ${\bf{u}}$.  However, the exponential map is too computationally demanding to use in practice, so several alternative retractions have been proposed in the literature \cite{absil2012projection,absil2015low,gawlik2018high}.  A particularly convenient class of retractions is based on the concept of a \emph{retractor}, which is formally defined below for our problem setting.

\begin{framed}
\begin{definition}[Retractor]
A smooth mapping $\mathcal{R}: T\mathcal{M} \rightarrow \text{Gr}\left(1,\mathbb{R}^{d_\mu}\right)$ is said to be a \emph{retractor} if, for all $\bm{\tau} \in \mathcal{M}$, $\mathbb{R}^{d_\mu} = \mathcal{R}_{\bm{\tau}}(\bm{0}) \oplus T_{\bm{\tau}} \mathcal{M}$ where $\mathcal{R}_{\bm{\tau}}$ is the restriction of $\mathcal{R}$ to $\bm{\tau}$ (i.e. $\mathcal{R}_{\bm{\tau}}({\bf{u}}) := \mathcal{R}(\bm{\tau},{\bf{u}})$).
\label{def:retractor}
\end{definition}
\end{framed}

Note in the above definition that $\text{Gr}\left(1,\mathbb{R}^{d_\mu}\right)$ denotes the \textit{Grassmann manifold} of all lines through the origin in $\mathbb{R}^{d_\mu}$ \cite{boothby1986introduction}.  Thus, a retractor $\mathcal{R}$ maps a point $\bm{\tau} \in \mathcal{M}$ and tangent vector ${\bf{u}} \in T_{\bm{\tau}} \mathcal{M}$ to a line in Euclidean space.  Given $\bm{\tau} \in \mathcal{M}$ and ${\bf{u}} \in T_{\bm{\tau}} \mathcal{M}$, define the affine space $\mathcal{A}_{\bm{\tau}}({\bf{u}}) = \bm{\tau} + {\bf{u}} + \mathcal{R}_{\bm{\tau}}({\bf{u}})$.  Then we can find the closest point $\bm{\tau}^* \in \mathcal{M} \cap \mathcal{A}_{\bm{\tau}}({\bf{u}})$ to the point $\bm{\tau} + {\bf{u}}$.  Conveniently, this operation yields a retraction $R$ which satisfies:
\begin{align}
R_{\bm{\tau}}({\bf{u}}) = \argmin_{\bm{\tau}^* \in \mathcal{M} \cap \mathcal{A}_{\bm{\tau}}({\bf{u}})} \frac{1}{2} | \bm{\tau} + {\bf{u}} - \bm{\tau}^* |^2.
\end{align}
We call $R$ the retraction \textit{induced} by the retractor $\mathcal{R}$ \cite{absil2012projection}.  Given that retractors are much easier to construct (and compute) than retractions, we elect to induce retractors from retractions in this paper.  In particular, we will employ the retraction induced by the retractor:
\begin{align}
\bm{v}_{\bm{\tau}}(\bm{\eta}) = \left\{ \begin{array}{rl}
\text{span} \left( (\bm{\tau} + \bm{\eta}) - \bm{\tau}_\text{min} \right),& \mathcal{G}(\bm{\tau} + \bm{\eta}) \ge \mathcal{Q}_\text{allow}\\
\text{span} \left( \bm{\tau}_\text{max} - (\bm{\tau} + \bm{\eta}) \right),& \mathcal{G}(\bm{\tau} + \bm{\eta}) < \mathcal{Q}_\text{allow}
\end{array} \right.
\label{eqn:retractor}
\end{align}
where $\bm{\tau}_\text{max}$ and $\bm{\tau}_\text{min}$ encode upper and lower bounds on the allowable tolerance variable as discussed previously in Subsection \ref{section:hyperrectangleConstruct}.  Pseudocode for the retraction induced by the above retractor is provided in Algorithm 1.

\begin{algorithm}[t!]
\caption{Manifold retraction operation}
\begin{algorithmic}[1]
\Function{manifoldRetraction}{$\bm{\tau}, \bm{\eta}$}
\If{$\mathcal{G}(\bm{\tau} + \bm{\eta}) \ge \mathcal{Q}_\text{allow}$} \Comment{$\bm{\tau} + \bm{\eta}$ is above manifold}
\State $\bm{v} = (\bm{\tau} + \bm{\eta}) - \bm{\tau}_\text{min}$ \Comment{Retractor per Eq. \eqref{eqn:retractor}}
\Else \Comment{$\bm{\tau} + \bm{\eta}$ is below manifold}
\State $\bm{v} = \bm{\tau}_\text{max} - (\bm{\tau} + \bm{\eta})$ \Comment{Retractor per Eq. \eqref{eqn:retractor}}
\EndIf
\State $\bm{\tau}_{\bm{v}}(s) = s \bm{v} + (\bm{\tau} + \bm{\eta})$ \Comment{Line from tangent space along retractor to manifold}
\State Find $s^*$ such that $\mathcal{G}(\bm{\tau}_{\bm{v}}(s^*)) = \mathcal{Q}_\text{allow}$ \Comment{Determine manifold intersection}\\
\Return $\bm{\tau}_{\bm{v}}(s^*)$
\EndFunction
\end{algorithmic}
\end{algorithm}

With a suitable retraction in hand, we can now discuss the proposed manifold traversal algorithm for solving the optimal tolerance allocation problem.  The algorithm begin with an initial guess, $\bm{\tau}_0 \in \mathcal{M}$, which is obtained in a manner analogous to retraction.  We begin at the origin, or $\bm{\tau}_\text{min}$, and construct a ray from this point along the direction of $\nabla_{\bm{\tau}} \mathcal{F}(\bm{\tau}_\text{min})$. We then traverse this ray until the point of manifold intersection which is $\bm{\tau}_0$, the initial guess to the traversal algorithm. From this point, the remainder of the algorithm is iterative until convergence so for the sake of generality we employ the notation of $\bm{\tau}_i$.

At the point $\bm{\tau}_i$, we construct a basis for the tangent space $T_{\bm{\tau}_i} \mathcal{M}$, and using this basis, we compute the orthogonal projection matrix ${\bf P}$ onto the tangent space. We then determine if $\bm{\tau}_i$ lies on the one of the walls of the tolerance bounding box $\mathcal{T}_{\text{bounding}}$ and, if so, check if $\nabla_{\bm{\tau}} \mathcal{F}(\bm{\tau}_i)$ has increasing normal derivative out of that wall. If both of these conditions are met, we zero out the corresponding row of the projection matrix ${\bf P}$, effectively projecting the tangent space onto the boundary of $\mathcal{T}_{\text{bounding}}$. This ensures that our algorithm does not prematurely terminate on the walls of the tolerance bounding box.  Pseudocode for this procedure is provided in Algorithm 2.

\begin{algorithm}[t]
\caption{Build projection matrix}
\begin{algorithmic}[1]
\Function{buildProjection}{$\bm{\tau},\bm{\tau}_\text{min},\bm{\tau}_\text{max}$} 
\State{$\text{CG} = 1$} 
\State Compute orthonormal basis $\left\{ {\bf t}_i \right\}_{i=1}^{d_{\bm{\mu}} - 1}$ of $\text{span}\left\{ \nabla_{\bm{\tau}} \mathcal{G}(\bm{\tau}) \right\}^{\perp}$ \Comment{Compute tangent space basis}
\State ${\bf T} = \left[ {\bf t}_1 | {\bf t}_2 | \cdots | {\bf t}_{d_{\bm{\mu}}-1}\right]$ \Comment{Create tangent space basis matrix}
\State ${\bf P} = {\bf T} {\bf T}^T$ \Comment{Initialize projection matrix}
\For{$k=1,2,\ldots,d_{\bm{\mu}}$}
\If{$\left(\bm{\tau}\right)_k == \left(\bm{\tau}_\text{min}\right)_k$ {\bf or} $\left(\bm{\tau}\right)_k == \left(\bm{\tau}_\text{max}\right)_k$} \Comment{Check if $\bm{\tau}$ is on wall}
\If{$ \nabla_{\bm{\tau}} \mathcal{F}(\bm{\tau}) \cdot {\bf n}_k \ge 0$} \Comment{${\bf n}_k$ is unit outward normal to wall $k$}
\State{$\text{CG} = 0$} \Comment{New wall intersection}
\For{$m=1,2,\ldots,d_{\bm{\mu}}$}
\State ${\bf P}_{km} = 0$ \Comment{Zero out $k^{th}$ row of projection matrix}
\EndFor
\EndIf
\EndIf
\EndFor\\

\Return{$\left[{\bf P},\text{CG}\right]$}
\EndFunction
\end{algorithmic}
\end{algorithm}

\begin{algorithm}[t]
\caption{Bound-constrained manifold gradient ascent}
\begin{algorithmic}[1]
\Function{manifoldGradientAscent}{$\bm{\tau}_0, \bm{\tau}_\text{min}, \bm{\tau}_\text{max}$}
\State{$i=0$}
\While{$i < N$}
\State{${\bf P} = \textproc{buildProjection}(\bm{\tau}_i,\bm{\tau}_\text{min},\bm{\tau}_\text{max})$} \Comment{Get projection matrix at $\bm{\tau}_i$}
\State{$\bm{v} = {\bf P} \left( \nabla_{\bm{\tau}} \mathcal{F}(\bm{\tau}_i) \right)$} \Comment{Compute projection of $\nabla_{\bm{\tau}} \mathcal{F}(\bm{\tau}_i)$}
\State Find $\alpha^*$ = $\displaystyle \argmax_{\substack{\alpha \in \mathbb{R} \text{ such that}\\ \bm{\tau}_i + \alpha \bm{v} \in \mathcal{T}_{\text{bounding}}}} \mathcal{F}\left(R_{\bm{\tau}_i}\left(\alpha \bm{v}\right)\right)$ \Comment{Line search for optimal step in ascent direction}
\State{$\bm{\tau}_{i+1} = R_{\bm{\tau}_i}\left(\alpha^* \bm{v}\right)$} \Comment{Compute $\bm{\tau}_{i+1}$}
\State{$i=i+1$} \Comment{Increment counter}
\EndWhile
\EndFunction
\end{algorithmic}
\end{algorithm}

\begin{algorithm}[t!]
\caption{Bound-constrained manifold nonlinear conjugate gradients}
\begin{algorithmic}[1]
\Function{manifoldConjugateGradients}{$\bm{\tau}_0, \bm{\tau}_\text{min}, \bm{\tau}_\text{max}$}
\State{$i=0$}
\While{$i < N$}
\State{$\left[{\bf P},\text{CG}\right] = \textproc{buildProjection}(\bm{\tau}_i,\bm{\tau}_\text{min},\bm{\tau}_\text{max})$} \Comment{Get ${\bf P}$ and check for new wall intersection}
\State{$\bm{g}^{(i)} = {\bf P} \left( \nabla_{\bm{\tau}} \mathcal{F}(\bm{\tau}_i) \right)$} \Comment{Compute projection of $\nabla_{\bm{\tau}} \mathcal{F}(\bm{\tau}_i)$}
\If{$\text{CG}$ {\bf and} $i>0$}
\State{$\beta_i = \frac{ \left( \bm{g}^{(i)} \right)^T \left( \bm{g}^{(i)} \right)}{\left( \bm{g}^{(i-1)} \right)^T \left( \bm{g}^{(i-1)} \right)}$} \Comment{Fletcher-Reeves search update weighting}
\State $\bm{v}^{(i)} = \bm{g}^{(i)} + \beta_{i}\mathcal{T}_{\alpha_{i-1}\bm{v}^{(i-1)}}(\bm{v}^{(i-1)})$ \Comment{Add vector transport to search direction}
\Else
\State $\bm{v}^{(i)} = \bm{g}^{(i)}$ \Comment{Perform gradient ascent}
\EndIf
\State Find $\alpha_i$ = $\displaystyle \argmax_{\substack{\alpha \in \mathbb{R} \text{ such that}\\ \bm{\tau}_i + \alpha \bm{v}^{(i)} \in \mathcal{T}_{\text{bounding}}}} \mathcal{F}\left(R_{\bm{\tau}_i}\left(\alpha \bm{v}^{(i)}\right)\right)$ \Comment{Line search for optimal step in ascent direction}
\State{$\bm{\tau}_{i+1} = R_{\bm{\tau}_i}\left(\alpha_i \bm{v}^{(i)}\right)$} \Comment{Compute $\bm{\tau}_{i+1}$}
\State{$i=i+1$} \Comment{Increment counter}
\EndWhile
\EndFunction
\end{algorithmic}
\end{algorithm}

Once the projection matrix ${\bf P}$ has been constructed, either the manifold gradient ascent method or the manifold conjugate gradient method may be employed to iterate the solution.  Pseudocode for the manifold gradient ascent method is provided in Algorithm 3, while pseudocode for the manifold conjugate gradient method is provided in Algorithm 4.  It should be noted that the algorithms largely follow the manifold gradient ascent and conjugate gradient algorithms provided in \cite{absil2009optimization}, though our algorithms also constrain the obtained solution to lie within the tolerance bounding box $\mathcal{T}_{\text{bounding}}$.  The manifold gradient ascent method chooses the ascent direction by projecting $\nabla_{\bm{\tau}} \mathcal{F}(\bm{\tau}_i)$ onto the tangent basis while the manifold conjugate gradient method uses the same ascent direction but additionally accounts for previous iterate search directions.   Note in our algorithm we have employed the somewhat standard Fletcher-Reeves weighting scheme for $\beta_i$, the search direction update \cite{fletcher1964function}. However there exists possible alternative choices of this parameter, such as Polak-Ribi\`ere which weights the search update by the change in magnitude between the iterates \cite{polak1969note}. A line search is then performed over this search direction, between the walls of $\mathcal{T}_{\text{bounding}}$, and the maximal value is set to $\bm{\tau}_{i+1}$. This line search process utilizes the retraction algorithm provided in Algorithm 1 since every function evaluation on the manifold is equivalent to an evaluation of a retracted tangent space entity.  For the examples shown in this paper, the line search process is carried out using Brent's method \cite{brent2013algorithms}.  Note that the first iteration of the manifold conjugate gradient method is simply manifold gradient ascent, and every time the algorithm reaches the boundary of $\mathcal{T}_{\text{bounding}}$, the manifold conjugate gradient method is restarted.  It should also be noted that the manifold conjugate gradient method exploits the notion of \textit{vector transport} of $\bm{\xi} \in T_{\bm{\tau}} \mathcal{M}$ along ${\bf{u}}  \in T_{\bm{\tau}} \mathcal{M}$ from $\bm{\tau} \in \mathcal{M}$ through the differentiated retraction operator:
\begin{align}
\mathcal{T}_{{\bf{u}}}(\bm{\xi}) := D R_{\bm{\tau}}({\bf{u}})[\bm{\xi}]
\end{align}
where $D$ denotes the Fr\'{e}chet derivative operator.  For more details on vector transport, the reader is referred to Chapter 8 of \cite{absil2009optimization}.  The manifold gradient ascent method is technically first-order, while the manifold conjugate gradient method may be regarded as a blend between the first-order gradient ascent method and a second-order Newton method.  The manifold gradient ascent method suffers from slow convergence in the presence of large disparity between Hessian eigenvalues, while the manifold conjugate gradient method is typically more robust \cite{absil2009optimization}.  It should finally be noted that Newton methods \cite{smith1994optimization}, quasi-Newton methods \cite{huang2015broyden,huper2004newton}, and trust-region methods \cite{absil2007trust} have also been developed for manifold optimization.

\section{Numerical Tests}
\label{sec:NumTests}

In this section, we apply the aforementioned methodology for optimal tolerance allocation to the setting of linear elasticity. We consider a suite of problems including: (i) a plate with a hole described by two design parameters, (ii) a plate with a hole described by six design parameters, and (iii) an L-Bracket described by seventeen design parameters. These problems are chosen to demonstrate the robustness and effectiveness of the tolerance allocation algorithm with respect to dimensionality over complex and intricate geometric configurations. 

In each of our numerical tests, we begin by constructing the appropriately-sized sampling domain as discussed in Subsection \ref{section:hyperrectangleConstruct}. From here, we are capable of building surrogate models over a variety of polynomial degrees and ranks using the separated representation methodology presented in Subsection \ref{section:SRConstruct}. These models are constructed from a database of realizations from uniformly-distributed Monte Carlo samples corresponding to geometries which reside in the predetermined sampling domain. Both the average
\begin{equation}
\|{\bf e}_{r,p}\|_\text{M} := \frac{1}{N_\text{c}} \sum_{i=1}^{N_\text{c}} \left|\frac{ \mathcal{Q}(\bm{\mu}_i) - \tilde{\mathcal{Q}}_{r,p}(\bm{\mu}_i) }{ \mathcal{Q}(\bm{\mu}_i)}\right|
\label{eqn:mean_norm}
\end{equation}
and maximum
\begin{equation}
\|{\bf e}_{r,p}\|_\infty := \max_{1 \le i \le N_c} \left|\frac{ \mathcal{Q}(\bm{\mu}_i) - \tilde{\mathcal{Q}}_{r,p}(\bm{\mu}_i) }{ \mathcal{Q}(\bm{\mu}_i)}\right|
\label{eqn:max_norm}
\end{equation}
relative errors in the surrogate models, as a function of polynomial degree $p$ and separation rank $r$, are considered in this section. Here, $N_c$ is the number of compared samples, none of which are used in the construction of $\tilde{\mathcal{Q}}_{r,p}(\bm{\mu})$. The average relative error provides a notion of surrogate model convergence while the maximum error is the pointwise quantity we wish to accurately capture, since the worst member in the set of designs indicates compliance to the system performance. Tables containing these errors will be presented and leveraged in our choice of surrogate model construction. Moreover, our proposed methodology allows the user to effectively tune the surrogate model to be within their desired fidelities through this approach. 

Throughout these numerical tests, we allocate tolerances while considering the effect of design parameter variations on the maximum von Mises stress at specified areas, effectively characterizing part failure. For the plate with a hole described by two design parameters, we also consider the total strain energy of the design configuration, providing a notion of overall geometric stiffness. To assess our tolerance allocation algorithm's effectiveness, we consider the three following measures throughout our numerical results. First, we consider
\begin{equation}
\bm{\epsilon}_{r,p} := \hat{\bm{\tau}} - \bm{\tau}_{r,p}
\end{equation}
which is the error between the obtained tolerance of a low-fidelity surrogate model and the \emph{true optimal} value, which comes from either a dense sampling or a high-fidelity surrogate model. In particular, we examine $\| \bm{\epsilon}_{r,p} \|_\infty$.  When the manifold is convex, we expect a single global maximum, so provided the low-fidelity surrogate model $\tilde{\mathcal{Q}}_{r,p}(\bm{\mu})$ converges to the system performance $\mathcal{Q}(\bm{\mu})$ in a pointwise manner, we expect convergence with respect to this measure.  In the non-convex setting, multiple global maxima may exist, so this first measure of error may not be appropriate.  Second, we consider

\begin{equation}
\varphi_A\left(  \bm{\tau}_{r,p} \right) := \frac{\left| \mathcal{F}_A\left(  \hat{\bm{\tau}} \right) - \mathcal{F}_A\left(  \bm{\tau}_{r,p} \right) \right|}{\mathcal{F}_A\left(  \hat{\bm{\tau}} \right)}
\end{equation}

which is the relative error in the objective functional with respect to the true optimal value. If $A = 1$, this corresponds to the 1-norm, Eq. \eqref{eqn:tolNorm_1}; if $A = \bm{\mu}$, this corresponds to the $\bm{\mu}$-norm, Eq. \eqref{eqn:tolNorm_mu}; and if $A = -1$, this corresponds to the $-1$-norm, Eq. \eqref{eqn:tolNorm_-1}.  Convergence can be attained with respect to this measure even in the non-convex setting where multiple global maxima may exist since these global maxima return the same value for the objective functional.  However, our algorithm may return local maxima rather than global maxima in the non-convex setting.  We lastly consider the following measure

\begin{equation}
\gamma_A\left(  \bm{\tau}_{r,p} \right) := \frac{\left| \mathcal{Q}_\text{A,allow} - \mathcal{G}_{A}(\bm{\tau}_{r,p}) \right|}{\mathcal{Q}_\text{A,allow}}
\end{equation}

which is the relative error in the constraint functional with respect to the performance constraint. If $A = SE$, this corresponds to the strain energy measure, and if $A = M$, this corresponds to the maximum von Mises stress measure, as defined later in \eqref{eqn:max_stress}. This measure assesses the convergence of the surrogate model to the the true model in the optimization routine. We do expect convergence with respect to this measure, since it is our manifold definition. Before proceeding with the presentation of numerical results, we briefly discuss linear elasticity theory, in its parametric form, which is employed throughout the remainder of this section.

\subsection{Linear Elasticity}


In linear elasticity, the components of the infinitesimal strain tensor are given by the symmetric part of the gradient of the displacement field:
\begin{align}
\bm{\varepsilon}(\vec{u}) = \frac{1}{2} \left(\left(\nabla \vec{u}\right) + \left(\nabla \vec{u}\right)^T \right),
\label{eqn:linGL}
\end{align}
where $\vec{u}$ is displacement field. With an appropriate material model ${\bf D}$, we can relate the internal stresses to the strain via:
\begin{align}
\bm{\sigma} = {\bf D} : \bm{\varepsilon}.
\end{align}
The strong form of the parametric PDEs governing linear elasticity are given in terms of internal stresses by:

For every $\bm{\mu} \in \mathcal{D}_{\hat{\bm{\mu}}}$, find $\vec{u}({\bf x}) \in C^2(\Omega_{\bm{\mu}})$ such that:
\begin{align}
(S) \left\{ \begin{array}{rll}
\nabla \cdot \bm{\sigma} &= \vec{f},& {\bf x} \in \Omega_{\bm{\mu}}\\
\vec{u} &= \vec{g} ,& {\bf x} \in \Gamma_{\bm{\mu},D}\\
\bm{\sigma} \cdot \vec{n}_{\bm{\mu}} &= \vec{h} ,& {\bf x} \in \Gamma_{\bm{\mu},N},
\end{array} \right.
\end{align}

where $\vec{f}$ is the external loading, $\vec{g}$ is the Dirichlet boundary condition over the parametric Dirichlet boundary $\Gamma_{\bm{\mu},D}$, and $\vec{h}$ is the Neumann boundary condition over the parametric Neumann boundary $\Gamma_{\bm{\mu},N}$. Note that here, $\vec{n}_{\bm{\mu}}$ is the the outward normal director of $\Gamma_{\bm{\mu},N}$. Note that our methodology is capable of additionally parameterizing external forcing, boundary conditions, and material constants; however we restrict ourselves here simply to geometric variations.

We seek the weak solution to this problem by invoking the \textit{principle of virtual work}. In particular, we consider the strain field $\bm{\varepsilon}(\vec{u})$ which minimizes the potential energy configuration of the system against a space of test functions. The \emph{trial} and \emph{test} spaces for the weak linear elasticity problem are defined in their parametric form as:
\[
\mathcal{S}(\bm{\mu}) := \left\{ \vec{u} : \Omega_{\bm{\mu}} \rightarrow \R^{d_s} \ \big| \ \vec{u} \in \left(\mathcal{H}^1(\Omega_{\bm{\mu}})\right)^{d_s} \text{ and} \left. \vec{u} \right|_{\Gamma_{\bm{\mu},D}} = \vec{g} \right\}
\]
and
\[
\mathcal{V}(\bm{\mu}) := \left\{ \vec{w} : \Omega_{\bm{\mu}} \rightarrow \R^{d_s} \ \big| \ \vec{w} \in \left(\mathcal{H}^1(\Omega_{\bm{\mu}})\right)^{d_s} \text{ and} \left. \vec{w} \right|_{\Gamma_{\bm{\mu},D}} = 0 \right\}
\]
respectively. Then, the variational form of this parametric PDE system is given by the $L^2$ inner product with an arbitrary test function $\vec{w} \in \mathcal{V}(\bm{\mu})$ followed by an integration by parts. This allows us to express the system in variational form as:

For $\bm{\mu} \in \mathcal{D}_{\hat{\bm{\mu}}}(\bm{\tau})$, find $\vec{u} \in \mathcal{S}(\bm{\mu})$ such that:
\[
a\left(\vec{w},\vec{u};\bm{\mu} \right) = \ell(\vec{w};\bm{\mu}) \hspace{20pt} \forall \ \vec{w} \in \mathcal{V}(\bm{\mu})
\]
where:
\[
a(\vec{w},\vec{u};\bm{\mu}) = \int_{\Omega_{\bm{\mu}}} \bm{\varepsilon}(\vec{w}) : {\bf D} : \bm{\varepsilon}(\vec{u}) \ d\Omega_{\bm{\mu}} \hspace{20pt} \forall \ \vec{w} \in \mathcal{V}(\bm{\mu})
\]
and
\[
\ell(\vec{w};\bm{\mu}) = \int_{\Omega_{\bm{\mu}}} \vec{w} \cdot \vec{f} \ d\Omega_{\bm{\mu}} + \int_{\Gamma_{N_{\bm{\mu}}}} \vec{w} \cdot \vec{h} \ d \Gamma_{\bm{\mu},N}.\hspace{20pt} \forall \ \vec{w} \in \mathcal{V}(\bm{\mu})
\]

For isogeometric implementation, we must convert the above weak formulations into a system of algebraic equations. This is accomplished through the application of Galerkin's method where we work in the finite-dimensional subspaces $\mathcal{S}^h(\bm{\mu}) \subset \mathcal{S}(\bm{\mu})$ and $\mathcal{V}^h(\bm{\mu}) \subset \mathcal{V}(\bm{\mu})$. These spaces are defined using the NURBS basis scaled by vector-valued control variables. In particular, the trial and test spaces are defined as:
\[
\mathcal{S}^h(\bm{\mu}) := \left\{ \vec{u}^h \in \mathcal{S}(\bm{\mu}) \colon  \vec{u}^h({\bf x}) = \sum_i \vec{d}_i N_i ({\bf x}) \right\};
\]
\[
\mathcal{V}^h(\bm{\mu}) := \left\{ \vec{w}^h \in \mathcal{S}(\bm{\mu}) \colon  \vec{w}^h({\bf x}) = \sum_i \vec{c}_i N_i ({\bf x}) \right\},
\]
where we note that the splines in this space must be at least $C^0$-continuous. To obtain the Galerkin form of the parametric PDE system, we analogously perform the $L^2$ inner product between members of these finite-dimensional test and trial spaces:

For $\bm{\mu} \in \mathcal{D}_{\hat{\bm{\mu}}}(\bm{\tau})$, find $\vec{u}^h \in \mathcal{S}^h(\bm{\mu})$ such that:
\[
a(\vec{w}^h,\vec{u}^h;\bm{\mu}) = \ell(\vec{w}^h;\bm{\mu}) \hspace{20pt} \forall \ \vec{w}^h \in \mathcal{V}^h(\bm{\mu}).
\]
This amounts to solving the matrix system Eq. \eqref{eqn:parametricPDE} for given $\bm{\mu} \in \mathcal{D}_{\hat{\bm{\mu}}}(\bm{\tau})$ where:
\begin{align}
\left[ \textbf{K}(\bm{\mu}) \right]_{PQ} = a(N_i \hat{e}_A,N_j \hat{e}_B;\bm{\mu}) \hspace{20pt} \text{and} \hspace{20pt} \left[ \textbf{F}(\bm{\mu}) \right]_{P} = \ell(N_i \hat{e}_A;\bm{\mu}),
\end{align}
where $P,Q$ are associated with an indexing scheme returning a global row number for each degree of freedom $A$ and basis function $i$. See \cite[Chapter 2]{HughesFEM} for more details.

From the solution vector $\vec{u}$, we can construct the surrogate models to the aforementioned system performances we consider throughout the numerical tests. In particular, we consider:
\begin{align}
\text{Maximum von Mises Stress:} \hspace{5pt} \mathcal{Q}_\text{M} (\bm{\mu}) = \max_{\bm{\xi}\in \mathcal{P}} \sigma_v({\bf x(\bm{\xi})},\bm{\mu}) \hspace{30pt} \text{Strain Energy:} \hspace{5pt} \mathcal{Q}_\text{SE} (\bm{\mu}) = \int_{\Omega_{\bm{\mu}}} \bm{\sigma} : \bm{\varepsilon} \ d \Omega_{\bm{\mu}}
\label{eqn:strain_energy}
\end{align}
for a set $\mathcal{P} \subset \hat{\Omega}$ specified \textit{a priori}. Moreover, the corresponding performance measures are given by:
\begin{align}
\text{Maximum von Mises Stress:} \hspace{5pt} \mathcal{G}_\text{M}(\bm{\tau}) = \max_{\bm{\mu} \in \mathcal{D}_{\hat{\bm{\mu}}}(\bm{\tau})} \mathcal{Q}_{M}(\bm{\mu}) \hspace{30pt} \text{Strain Energy:} \hspace{5pt} \mathcal{G}_\text{SE}(\bm{\tau}) = \max_{\bm{\mu} \in \mathcal{D}_{\hat{\bm{\mu}}}(\bm{\tau})} \mathcal{Q}_{SE}(\bm{\mu}) 
\label{eqn:max_stress}
\end{align}

Finally, analogous to Eq. \eqref{eqn:OptProb}, we have in this setting the following optimization problem:

Given $\hat{\bm{\mu}}$, find $\hat{\bm{\tau}}$ such that
\[
\hat{\bm{\tau}} = \argmax_{\bm{\tau} \in \mathcal{T}_\text{allow}} \mathcal{F}(\bm{\tau}) \hspace{15pt} \text{where} \hspace{15pt}
\mathcal{T}_\text{allow} := \left\{ \bm{\tau} \in \R^{d_{\bm{\mu}}} : \mathcal{G}_\text{A}(\bm{\tau}) \le \mathcal{Q}_\text{allow} \right\}
\]
where $A = M$ or $SE$ in the case of maximum von Mises stress and strain energy, respectively.


\subsection{Plate with a Hole with Two Design Parameters}
\label{sec:2dplate}

\begin{figure}[t!]
	\centering
	\begin{subfigure}[c]{.45\textwidth}
		\centering
		\includegraphics[scale=1.19]{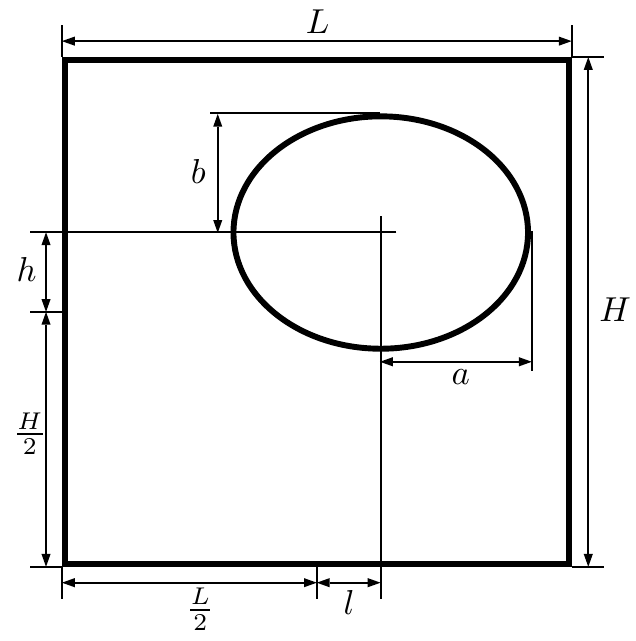}
	\end{subfigure}
	\begin{subfigure}[c]{.45\textwidth}
		\centering
		\includegraphics[scale=1.20]{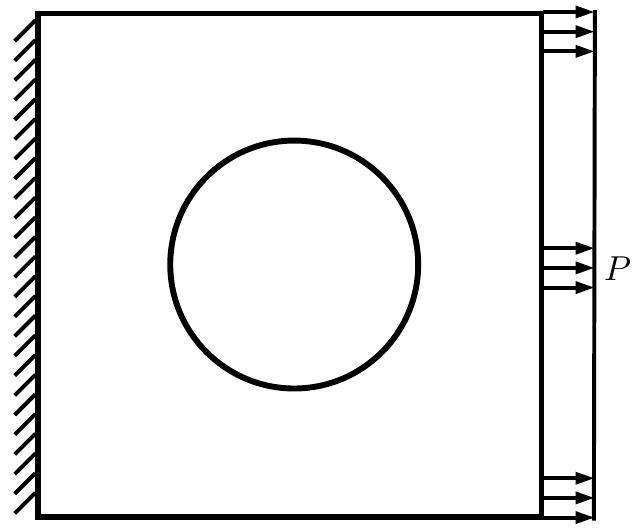}
	\end{subfigure}\\
	\begin{subfigure}[c]{.45\textwidth}
		\centering
		\includegraphics[scale=0.75]{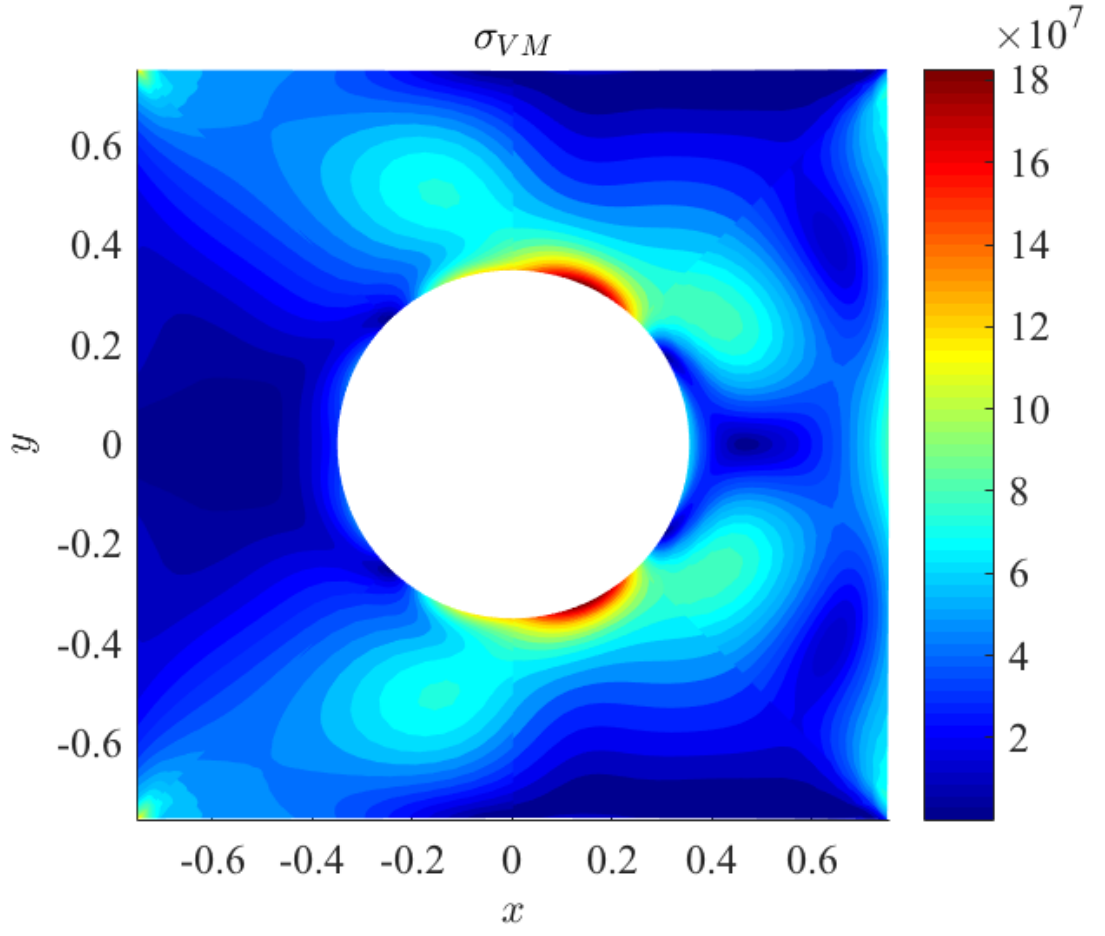}
	\end{subfigure}
	\caption{(upper left) Plate with a hole geometric configuration. (upper right) The loading and boundary conditions associated with the plate with a hole problem. A uniform loading of $P=30 \times 10^6$ is applied to the right and the left side of the plate has zero displacement boundary conditions. (bottom) von Mises stress distribution for the plate with a hole nominal configuration.}
	\label{fig:PlateHole_config}
\end{figure}

In this subsection, we consider the structural deformation of a plate with an elliptic hole whose major axes are aligned with the sides of the plate.  The geometric dimensions of the plate are illustrated in the upper left of Fig. \ref{fig:PlateHole_config}.  The length and height of the plate are taken to be $L = H = 1.5$, the radii of the hole are taken to be $a = b = 0.35$, and the displacement of the center of the hole from the center of the plate is expressed in terms of a dimensionless design variable $\bm{\mu} \in \mathbb{R}^2$ via:
\begin{align*}
l = \mu_1 \left( \frac{L}{2} - a \right) \hspace{10pt} \text{ and } \hspace{10pt} h = \mu_2 \left( \frac{H}{2} - b \right).
\end{align*}
The nominal configuration is with the hole centered in the plate, corresponding to:
\begin{equation*}
\hat{\bm{\mu}} = \left( \begin{array}{c}
\hat{\mu}_{1}\\
\hat{\mu}_{2} \end{array} \right) = \left( \begin{array}{c}
0\\
0 \end{array} \right).
\end{equation*}
The plate is assumed to be made of an isotropic material with Young's modulus $E = 200 \times 10^9$ and Poisson ratio $\nu = 0.3$, and the plate is assumed to be in a plane stress state. The loading and boundary conditions are depicted in the upper right of Fig. \ref{fig:PlateHole_config}.  In particular, a uniform loading of $P=30 \times 10^6$ is applied to the right side of the plate, zero displacement boundary conditions are applied at the left side of the plate, and zero traction boundary conditions are applied at the top and bottom sides of the plate as well along the hole. To obtain the structural deformation of the plate under the applied loading and boundary conditions, we discretize the plate with a 512-element, multi-patch isogeometric analysis model parametrized with quadratic NURBS functions. This choice of analysis model accurately represents the circular hole and additionally provides a natural parametric modeling framework for obtaining quantities of interest as a function of the design variable.  The von Mises stress distribution for the nominal configuration is displayed in the bottom of Fig. \ref{fig:PlateHole_config}.  Herein, we study the impact of variations in hole placement on (i) the maximum von Mises stress occurring at either the top or bottom of the hole and (ii) the strain energy of the plate configuration.

This problem is thoroughly investigated since, due to the low-dimensional nature of the design space, an ``exact'' optimal tolerance is obtainable through a dense sampling of the design space. Through this approach, we are capable of comparing the accuracy of the tolerance obtained through our tolerance allocation algorithm to this optimum as a function of polynomial degree and rank of the chosen surrogate. 

\begin{table}[t!] 
	 \centering 
	\caption{Plate with a hole with two design parameters: The strain energy and maximum von Mises stress surrogate modeling errors constructed from $N = 100$ samples and $N_c = 500$.}
	 \label{table:PlateHole2Dhk_SE} 	 	 
	 \begin{tabular}{cc|cccc|cccc}
	 & & \multicolumn{4}{c|}{Strain Energy} & \multicolumn{4}{c}{Maximum von Mises Stress} \\
& & $r=1$ & $r=2$ & $r=3$ & $r=4$ & $r=1$ & $r=2$ & $r=3$ & $r=4$ \\ 
\hhline{==========}
\multirow{7}{*}{$\| {\bf e }_{r,p} \|_\infty$} & degree 0 & 1.386e-1 & & & & 2.030e-1 & & & \\ 
& degree 1  & 6.861e-2  & 6.793e-2  &  & & 1.327e-2  & 8.633e-3  &  &  \\ 
& degree 2  & 1.309e-3  & 5.232e-4  & 5.326e-4  & & 8.071e-3  & 1.910e-3  & 2.236e-3  & \\ 
& degree 3  & 1.007e-3  & 8.625e-5  & 8.923e-5  & 8.965e-5 & 7.537e-3  & 1.602e-4  & 4.724e-4  & 4.969e-4  \\ 
& degree 4  &   & 1.206e-5  & 6.635e-6  & 6.841e-6 &  & 1.432e-4  & 2.918e-5  & 2.198e-4 \\ 
& degree 5  &   &  & 4.811e-6  & 1.438e-6 &  &  & 2.064e-5  & 2.014e-5 \\ 
& degree 6  &   &  &   & 2.288e-6 &  &  &  & 4.958e-6  \\ 
\hhline{==========}
\multirow{7}{*}{$\| {\bf e }_{r,p} \|_M$} & degree 0 & 5.194e-2 & & & & 6.777e-2 & & & \\ 
& degree 1  & 2.618e-2  & 2.623e-2  & & & 2.410e-3  & 2.159e-3  &   & \\ 
& degree 2  & 2.870e-4  & 1.478e-4  & 1.474e-4  & & 1.517e-3  & 4.241e-4  & 4.225e-4  & \\ 
& degree 3  & 2.617e-4  & 1.896e-5  & 1.893e-5  & 1.901e-5 & 1.542e-3  & 3.705e-5  & 3.946e-5  & 4.022e-5\\ 
& degree 4  &  & 2.131e-6  & 1.372e-6  & 1.375e-6 & & 1.532e-5  & 4.088e-6  & 6.448e-6 \\ 
& degree 5  &  &   & 2.401e-7  & 1.516e-7 &   &   & 1.654e-6  & 1.212e-6\\ 
& degree 6  & &   &   & 9.288e-8 &   &   &   & 3.343e-7\\ 
\hhline{==========}
	 	 \end{tabular}
	 \end{table}
	 
\begin{table}[!t] 
\centering
\caption{Plate with a hole with two design parameters: Optimal tolerance values obtained using a dense sampling of the sampling domain. These values are treated as the ``exact'' optima. Bold numbers indicate values lying on the boundary of the tolerance bounding box.}
\label{table:2dExactTol}
\begin{tabular}{c|ccc|ccc}
& \multicolumn{3}{c|}{Strain Energy} & \multicolumn{3}{c}{Maximum von Mises Stress} \\
& $\mathcal{F}_1(\bm{\tau})$ & $\mathcal{F}_{\bm{\mu}}(\bm{\tau})$ & $\mathcal{F}_{-1}(\bm{\tau})$ & $\mathcal{F}_1(\bm{\tau})$ & $\mathcal{F}_{\bm{\mu}}(\bm{\tau})$ & $\mathcal{F}_{-1}(\bm{\tau})$\\
\hline
$\hat{\tau}_1$ & 0.114 & {\bf 0.153} & 0.100 & {\bf 0.263} & {\bf 0.263} & 0.100 \\
$\hat{\tau}_2$ & 0.081 & {\bf 0.000} & 0.093 & {\bf 0.000} & {\bf 0.000} & 0.061 \\
\end{tabular}
\end{table}

The process begins by sizing the sampling domain in accordance with the techniques described in Subsection \ref{section:hyperrectangleConstruct}. The sizing process is accomplished using the performance constraints of $\mathcal{Q}_\text{M,allow} = 210 \times 10^6$, corresponding to approximately a $10\%$ deviation from the nominal maximum von Mises stress, and $\mathcal{Q}_\text{SE,allow} = 7.6 \times 10^6$, which corresponds to approximately a $10\%$ deviation from the nominal strain energy. Since we consider two separate quantities of interest, we must size two sampling domains according to this methodology. Given the nominal geometric configuration, this corresponds to the domains defined by:
\begin{equation*}
\begin{array}{ccc}
\bm{\tau}_\text{max,M} = \left( \begin{array}{c}
\tau_{\text{max,M},1}\\
\tau_{\text{max,M},2} \end{array} \right) = \left( \begin{array}{l}
0.263\\
0.098 \end{array} \right) \hspace{10pt} \text{ and } \hspace{10pt} & \bm{\tau}_\text{max,SE} = \left( \begin{array}{c}
\tau_{\text{max,SE},1}\\
\tau_{\text{max,SE},2} \end{array} \right) = \left( \begin{array}{l}
0.153\\
0.156\end{array} \right).
\end{array}
\end{equation*}
We also set the the minimum tolerances equal to zero, i.e. $\bm{\tau}_\text{min} = {\bf{0}}$.

To proceed with a demonstration of our methodology, we construct a surrogate model to the aforementioned quantities of interest. This is accomplished by constructing separated representations of the system performances presented in Eq. \eqref{eqn:strain_energy}. Determining the appropriate polynomial degree and rank amounts to performing a survey of these parameters and selecting the model which suits the desired fidelity. A set of $N = 100$ Monte Carlo samples are used in the construction of these surrogate models and their relative accuracy is computed using an additional $N_c = 500$ samples not used in the model construction.  It should be noted that we construct two separate surrogate models for the von Mises stress at the top and bottom of the hole since, while the von Mises stress at either of these points is smooth respect to design variable variations, the maximum von Mises stress among these two locations is not smooth.  We then take the maximum of the two values obtained from these surrogate models whenever we compute the maximum von Mises stress.  Table \ref{table:PlateHole2Dhk_SE} portrays the results of our survey and moreover exhibits the convergence of the separated representations with respect to the polynomial degree and rank of the expansion. Due to the least-squares nature of the separated representations, we only expect convergence in an $L^2$-sense. However, the smoothness associated with these response surfaces additionally provides convergence in Eq. \eqref{eqn:mean_norm} and Eq. \eqref{eqn:max_norm}. 

\begin{figure}[t!]
\centering
	\begin{subfigure}{\textwidth}
		\includegraphics{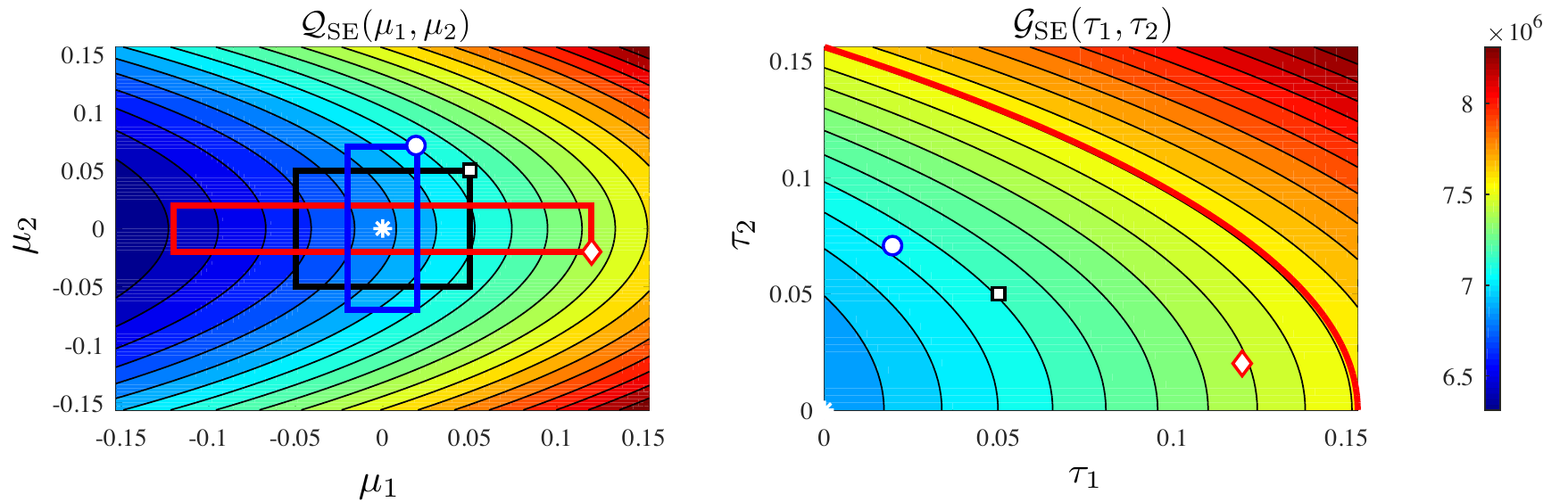}
		\caption{(left) A contour plot of the strain energy and (right) the performance measure $\mathcal{G}_\text{SE}(\tau_1,\tau_2)$.}
	\end{subfigure}
	\begin{subfigure}{\textwidth}
		\includegraphics{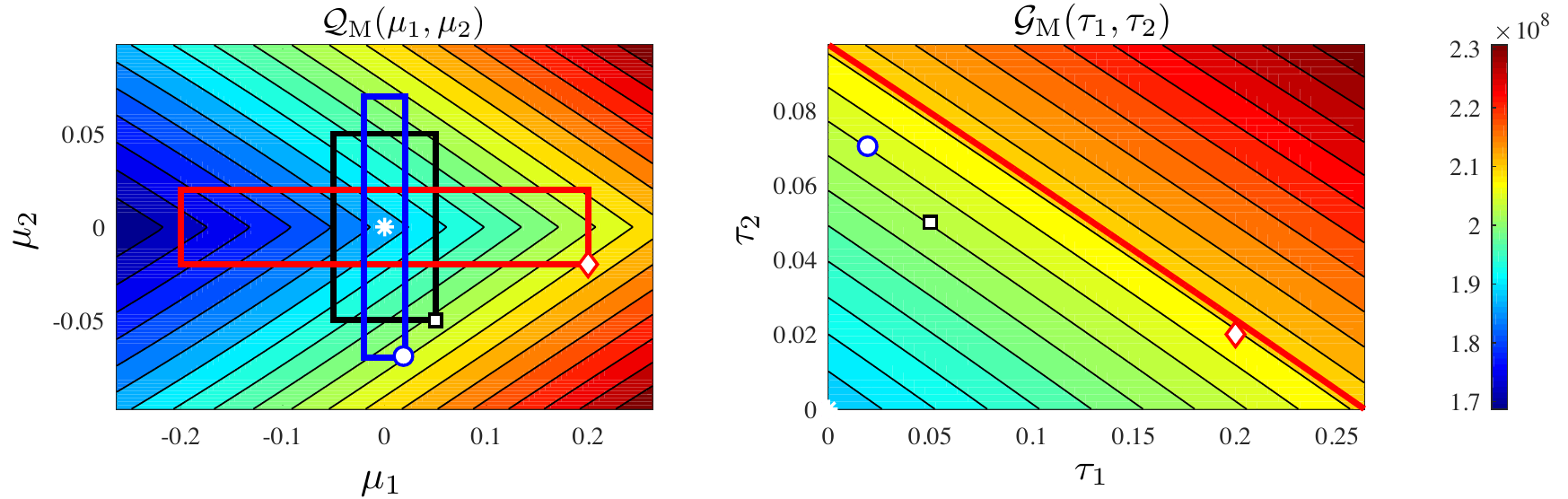}
		\caption{(left) A contour plot of the maximum von Mises stress and (right) the performance measure $\mathcal{G}_\text{M}(\tau_1,\tau_2)$.}
	\end{subfigure}
	\caption{Plate with a hole with two design parameters: The white asterisk at $\mu_1 = 0$ and $\mu_2 = 0$ denotes the nominal configuration of the plate with a hole with two design parameters. The three colored rectangles depict three different tolerance hyperrectangles, each corresponding to a tolerance $\bm{\tau}^{(i)}$, with hollow markers that indicate the location where the maximum of the restricted system performance $\mathcal{Q}(\mu_1,\mu_2) |_{\mathcal{D}_{\hat{\bm{\mu}}}(\bm{\tau}^{(i)})}$ is attained. The red contour line denotes the loci of $\bm{\tau}$, that is the immersed manifold, such that $\mathcal{G}(\tau_1,\tau_2) = \mathcal{Q}_\text{allow}$.}
	\label{fig:QSEandGSE}
\end{figure}

\begin{figure}[ht!]
\includegraphics{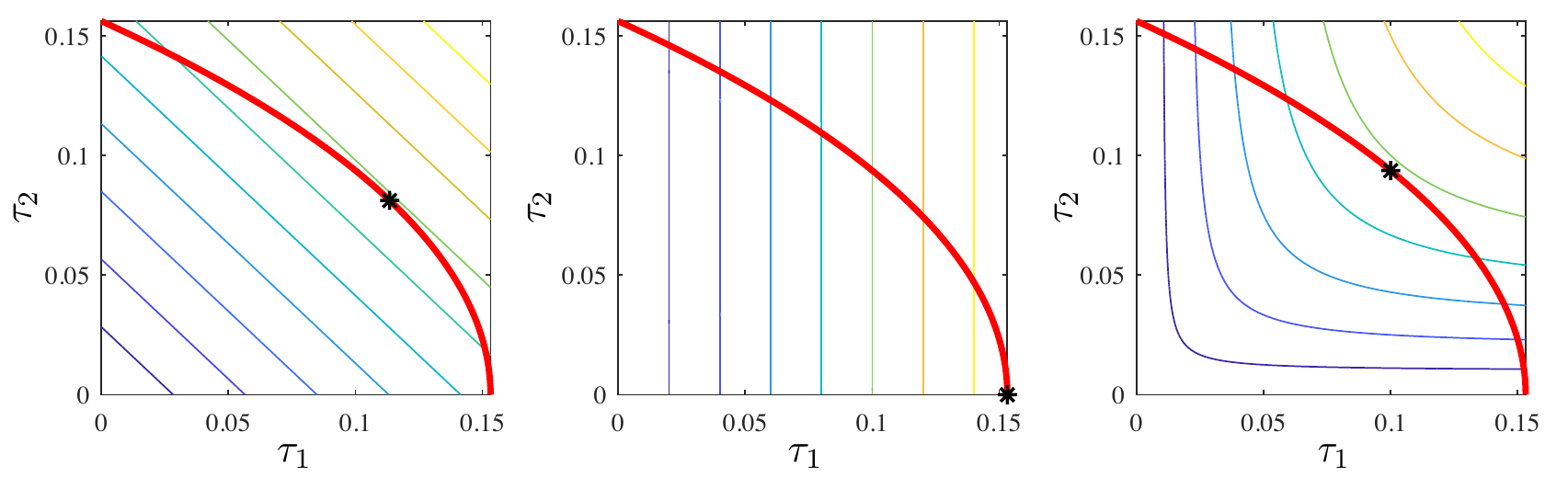}
\caption{Plate with a hole with two design parameters: The one-dimensional manifold corresponding to the level set of $\mathcal{G}(\bm{\tau}) = \mathcal{Q}_\text{allow,SE}$ for strain energy with (left) the tolerance measure $\mathcal{F}_1(\bm{\tau})$, (center) the tolerance measure $\mathcal{F}_{\bm{\mu}}(\bm{\tau})$, and (right) the tolerance measure $\mathcal{F}_{-1}(\bm{\tau})$. The black asterisk denotes the location of $\hat{\bm{\tau}}$ with respect to each norm.}
\label{fig:contours_1_norm_SE}
\end{figure}

\begin{figure}[ht!]
\includegraphics{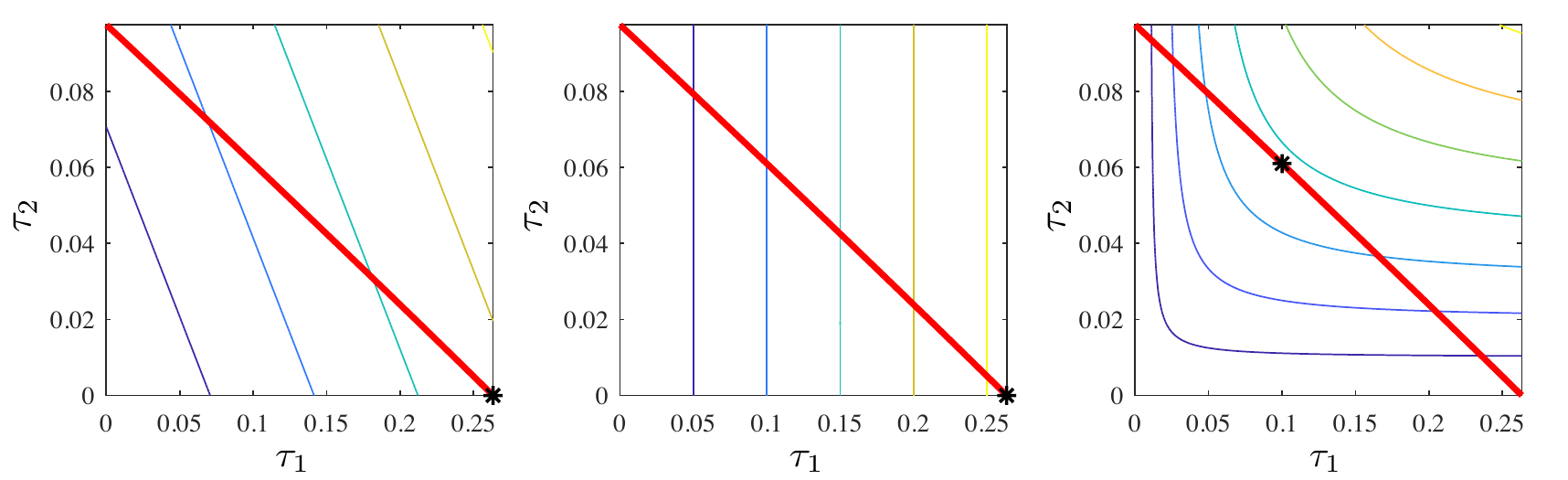}
\caption{Plate with a hole with two design parameters: The one-dimensional manifold corresponding to the level set of $\mathcal{G}(\bm{\tau}) = \mathcal{Q}_\text{allow,M}$ for maximum von Mises stress with (left) the tolerance measure $\mathcal{F}_1(\bm{\tau})$, (center) the tolerance measure $\mathcal{F}_{\bm{\mu}}(\bm{\tau})$, and (right) the tolerance measure $\mathcal{F}_{-1}(\bm{\tau})$. The black asterisk denotes the location of $\hat{\bm{\tau}}$ with respect to each norm.}
\label{fig:contours_1_norm_M}
\end{figure}
	 
With the surrogate models constructed using Eq. \eqref{eqn:strain_energy}, we are capable of employing our tolerance allocation algorithm. Fig. \ref{fig:QSEandGSE} depicts the response surfaces  for the system performances of strain energy and the maximum von Mises stress between the top and bottom of the hole as well as their corresponding performance measures restricted to the tolerance bounding box $\mathcal{T}_{\text{bounding}}$. Additionally, there are three representative tolerance hyperrectangles which are overlaid on the response surfaces along with the design maximizer, denoted with a hollow marker of identical color. The collection of these maximizers comprises the performance measure $\mathcal{G}(\bm{\tau})$. These markers are also shown on the figure depicting the performance measure for clarity. Moreover, the manifold of tolerances which attain the performance constraint is represented by the solid red line. This is the manifold over which our algorithm aims to maximize the available tolerance.

\begin{table}[t!] 
	 \centering 
	 \caption{Plate with a hole with two design parameters: Errors and convergence behavior of the tolerance allocation algorithm for the $1$-norm.}
	 \label{table:PlateHole2Dhk_1} 	 	 \begin{tabular}{cc|rrrrr|rrrrr}
$r$ & $p$ & $n_{GA}$ & $n_{CG}$ & $\| \bm{\epsilon}_{r,p} \|_\infty$ & $\varphi_1 \left(  \bm{\tau}_{r,p} \right)$ & $\gamma_\text{SE} \left(  \bm{\tau}_{r,p} \right)$ & $n_{GA}$ & $n_{CG}$ & $\| \bm{\epsilon}_{r,p} \|_\infty$ & $\varphi_1 \left(  \bm{\tau}_{r,p} \right)$ & $\gamma_\text{M} \left(  \bm{\tau}_{r,p} \right)$ \\
\hhline{============}
\multirow{3}{*}{1} & 1 & 2 & 2    & 7.462e-2    & 2.459e-1   &  6.568e-2  & 3 & 3    & 6.479e-4    &  2.455e-3   & 6.728e-4\\
& 2 & 3 & 3    & 2.311e-3    & 1.108e-3   & 1.664e-4  & 3 & 3    & 4.102e-3    &  1.583e-2   & 1.626e-3  \\
& 3 & 3 & 3    & 1.657e-3    &  1.324e-3 & 1.836e-4  & 4 & 4    & 3.156e-14    & 1.199e-13  & 5.677e-16  \\
\cline{1-12}
\multirow{3}{*}{2} & 2 & 3 & 3    & 8.492e-4    &  3.030e-4   & 3.764e-5  & 3 & 3    & 2.695e-3    & 1.034e-2    & 1.068e-3 \\
& 3 & 3 & 3    & 1.130e-4    &  5.113e-5  & 5.322e-6 & 3 & 3    & 2.477e-4    &  9.417e-4   & 9.812e-5   \\
& 4 & 3 & 3    & 9.023e-5    &  7.823e-6   & 2.463e-6  & 3 & 3    & 3.156e-14    &  1.199e-13   & 5.677e-16 \\
\cline{1-12} 
\multirow{3}{*}{3} & 3 & 3 & 3    & 1.212e-4    & 3.745e-5   & 3.520e-6  & 3 & 3    & 3.156e-14    &  1.199e-13  & 5.677e-16 \\
& 4 & 3 & 3    & 7.976e-5    &  9.094e-6 & 2.631e-6 & 3 & 3    & 3.156e-14    &  1.199e-13   & 5.677e-16  \\
& 5 & 3 & 3    & 7.949e-5    &  1.056e-5   & 3.720e-8  & 3 & 3    & 3.156e-14    &  1.199e-13 & 5.677e-16  \\
\cline{1-12}
\multirow{3}{*}{4} & 4 & 3 & 3    & 8.051e-5    &  1.060e-5   & 2.830e-6 & 3 & 3    & 3.156e-14    & 1.199e-13  & 5.677e-16  \\
& 5 & 3 & 3    & 7.938e-5    &  1.024e-5  & 7.970e-8 & 3 & 3    & 3.156e-14    &  1.199e-13 & 5.677e-16  \\
& 6 & 3 & 3    & 7.845e-5    &  1.092e-5   & 1.019e-8 & 3 & 3    & 3.156e-14    &  1.199e-13  & 5.677e-16   \\
\hhline{============}
 	 	 \end{tabular}
	 \end{table}
	 
	 \begin{table}[t!] 
	 \centering 
	 \caption{Plate with a hole with two design parameters: Errors and convergence behavior of the tolerance allocation algorithm for the $\bm{\mu}$-norm.}
	 \label{table:PlateHole2Dhk_mu} 	 	 \begin{tabular}{cc|rrrrr|rrrrr}
$r$ & $p$ & $n_{GA}$ & $n_{CG}$ & $\| \bm{\epsilon}_{r,p} \|_\infty$ & $\varphi_{\bm{\mu}} \left(  \bm{\tau}_{r,p} \right)$ & $\gamma_\text{SE} \left(  \bm{\tau}_{r,p} \right)$ & $n_{GA}$ & $n_{CG}$ & $\| \bm{\epsilon}_{r,p} \|_\infty$ & $\varphi_{\bm{\mu}} \left(  \bm{\tau}_{r,p} \right)$ & $\gamma_\text{M} \left(  \bm{\tau}_{r,p} \right)$ \\
\hhline{============}
\multirow{3}{*}{1} & 1 & 3 & 3    & 4.963e-2    & 4.792e-1   & 3.397e-2  & 3 & 3    & 6.479e-4    & 2.391e-3   & 6.728e-4  \\
& 2 & 8 & 8    & 1.769e-2    & 2.307e-5   & 1.329e-3   & 3 & 3    & 4.102e-3    & 1.583e-2   & 1.626e-3\\
& 3 & 9 & 9    & 1.558e-2    & 3.295e-5   & 1.030e-3  & 4 & 4    & 3.156e-14    & 1.156e-13   & 5.677e-16\\
\cline{1-12}
\multirow{3}{*}{2} & 2 & 9 & 9    & 1.097e-2    & 2.323e-5   & 5.111e-4  & 3 & 3    & 2.695e-3    & 1.034e-2   & 1.068e-3 \\
& 3 & 12 & 12    & 2.481e-3    & 3.165e-7   & 2.618e-5  & 3 & 3    & 2.477e-4    & 9.417e-4   & 9.812e-5  \\
& 4 & 13 & 13    & 1.581e-3    & 1.003e-7   & 1.065e-5 & 3 & 3    & 3.156e-14    & 1.161e-13   & 5.677e-16 \\
\cline{1-12}
\multirow{3}{*}{3} & 3 & 12 & 12    & 2.615e-3    & 9.234e-7   & 2.908e-5 & 3 & 3    & 3.156e-14    & 1.163e-13   & 5.677e-16  \\
& 4 & 13 & 13    & 1.324e-3    & 6.480e-8   & 7.474e-6  & 3 & 3    & 3.156e-14    & 1.162e-13   & 5.677e-16 \\
& 5 & 14 & 14    & 5.283e-4    & 1.636e-9   & 1.199e-6  & 3 & 3    & 3.156e-14    & 1.161e-13   & 5.677e-16  \\
\cline{1-12}
\multirow{3}{*}{4} & 4 & 13 & 13    & 1.308e-3    & 5.513e-8   & 7.290e-6  & 3 & 3    & 3.156e-14    & 1.162e-13   & 5.677e-16  \\
& 5 & 14 & 14    & 5.263e-4    & 1.060e-9   & 1.190e-6 & 3 & 3    & 3.156e-14    & 1.161e-13   & 5.677e-16  \\
& 6 & 14 & 14    & 3.716e-4    & 3.150e-9   & 5.961e-7 & 3 & 3    & 3.156e-14    & 1.161e-13   & 5.677e-16  \\
\hhline{============}
 	 	 \end{tabular}
	 \end{table}
	 
	  \begin{table}[t!] 
	 \centering 
	 \caption{Plate with a hole with two design parameters: Errors and convergence behavior of the tolerance allocation algorithm for the $-1$-norm.}
	 \label{table:PlateHole2Dhk_-1} 	 	 \begin{tabular}{cc|rrrrr|rrrrr}
$r$ & $p$ & $n_{GA}$ & $n_{CG}$ & $\| \bm{\epsilon}_{r,p} \|_\infty$ & $\varphi_{-1} \left(  \bm{\tau}_{r,p} \right)$ & $\gamma_\text{SE} \left(  \bm{\tau}_{r,p} \right)$ & $n_{GA}$ & $n_{CG}$ & $\| \bm{\epsilon}_{r,p} \|_\infty$ & $\varphi_{-1} \left(  \bm{\tau}_{r,p} \right)$ & $\gamma_\text{M} \left(  \bm{\tau}_{r,p} \right)$ \\
\hhline{============}
\multirow{3}{*}{1} & 1 & 2 & 2    & 2.908e-2    & 1.303e-1   & 2.322e-2 & 3 & 3    & 2.322e-3    & 1.081e-2   & 1.126e-3  \\
& 2 & 4 &  4   & 9.884e-4    & 6.032e-3   & 8.173e-4  & 3 & 3    & 5.294e-4    & 2.545e-3   & 2.517e-4\\
& 3 & 4 & 4    & 2.881e-4    & 1.545e-3   & 2.032e-4  & 3 & 3    & 1.098e-3    & 6.728e-3   & 7.024e-4  \\
\cline{1-12}
\multirow{3}{*}{2} & 2 & 4 & 4    & 1.062e-3    & 3.874e-3   & 5.193e-4 & 3 & 3    & 9.725e-4    & 5.157e-3   & 5.882e-4  \\
& 3 & 4 & 4    & 2.810e-4    & 2.710e-4   & 2.648e-5 & 3 & 3    & 3.815e-4    & 1.530e-4   & 3.920e-5\\
& 4 & 4 & 4    & 1.975e-5    & 6.544e-5   & 2.025e-5 & 3 & 3    & 3.138e-4    & 3.604e-5   & 2.646e-5\\
\cline{1-12}
\multirow{3}{*}{3} & 3 & 4 & 4    & 5.129e-4    & 3.360e-4   & 3.547e-5  & 3 & 3    & 3.730e-4    & 1.583e-4   & 3.978e-5  \\
& 4 & 4 & 4    & 1.320e-4    & 6.357e-4   & 7.673e-5  & 3 & 3    & 3.035e-4    & 4.272e-5   & 2.720e-5 \\
& 5 & 4 & 4    & 3.208e-5    & 1.927e-4   & 1.562e-5   & 3 & 3    & 3.053e-4    & 1.355e-5   & 2.401e-5\\
\cline{1-12}
\multirow{3}{*}{4} & 4 & 4 & 4    & 2.382e-4    & 6.842e-4   & 8.332e-5 & 3 & 3    & 3.036e-4    & 4.447e-5   & 2.739e-5 \\
& 5 & 4 & 4    & 2.434e-5    & 2.293e-4   & 2.068e-5 & 3 & 3    & 3.055e-4    & 1.451e-5   & 2.411e-5 \\
& 6 & 4 & 4    & 8.573e-6    & 8.146e-5   & 1.700e-7 & 3 & 3    & 3.037e-4    & 7.077e-6   & 2.331e-5  \\
\hhline{============}
 	 	 \end{tabular}
		 \vspace{20pt}
	 \end{table}

The manifolds arising from constraint equality in Fig. \ref{fig:QSEandGSE} are shown in Fig. \ref{fig:contours_1_norm_SE} and Fig. \ref{fig:contours_1_norm_M} with contours associated with the measures defined by Eq. \eqref{eqn:tolNorm_1}, Eq. \eqref{eqn:tolNorm_mu}, and Eq. \eqref{eqn:tolNorm_-1}, respectively, overlaid. The optimal tolerances with respect to these measures, i.e. $\hat{\bm{\tau}}$, are denoted by the black asterisk. Clearly, the location of this optimal tolerance is dependent on the choice of norm; however the traversal algorithm is agnostic with respect to this choice. Once again, the low-dimensional nature of this problem allows us to numerically determine the values of these optima. Therefore, we are able to assess the efficacy of the algorithm as a function of the polynomial degree and rank of the underlying separated representations.  The ``exact'' optimal tolerance values are tabulated in Table \ref{table:2dExactTol} for each tolerance measure and the subsequent allocation results are compared to these values.  Note that the optimal tolerances for maximum von Mises stress in both the 1-norm and the $\bm{\mu}$-norm are identical and reside on the boundary of the tolerance bounding box. On the other hand, the $-1$-norm has an isotropized tolerance which is almost centered in the tolerance bounding box.

Tables \ref{table:PlateHole2Dhk_1}, \ref{table:PlateHole2Dhk_mu}, and \ref{table:PlateHole2Dhk_-1} depict the effectiveness of our algorithm with respect to polynomial degree and rank of the surrogate model. In these tables, the error in the obtained tolerance, the relative error in the obtained objective functional, and the relative error in the obtained constraint functional are reported. As is clearly demonstrated, the accuracy of the obtained tolerance behaves similarly to the accuracy in the surrogate model construction. The columns $n_{GA}$ and $n_{CG}$ are the total number of iterations employed for manifold gradient ascent and manifold conjugate gradient, respectively, until the increase in allocated tolerance size is less than $10^{-6}$.  Rapid iterative convergence is realized for each case.  Note also that the manifold gradient ascent and manifold conjugate gradient methods require the same number of iterations for each case.  This is because the manifold is one-dimensional for the considered problem, so the search direction is the same for both methods during each iteration.


\subsection{Plate with a Hole with Six Design Parameters}

In this subsection, we consider the same problem as in Subsection \ref{sec:2dplate}, except that all the geometric dimensions of the plate are expressed in terms of a dimensionless design variable $\bm{\mu} \in \mathbb{R}^6$ via:
\begin{align*}
L = \mu_1 \mathscr{L}, \hspace{15pt} H = \mu_2 \mathscr{L}, \hspace{15pt} l = \mu_3 \left( \frac{L}{2} - a \right), \hspace{15pt} h = \mu_4 \left( \frac{H}{2} - b \right), \hspace{15pt} a = \mu_5 \mathscr{L}, \hspace{10pt} \text{ and } \hspace{10pt} b = a \left( 1- \left(\mu_6\right)^2 \right),
\end{align*}
where $\mathscr{L} = 1$ is a chosen length scale.  The nominal configuration for the plate is then:
\begin{equation*}
\hat{\bm{\mu}} = \left( \begin{array}{c}
\hat{\mu}_{1}\\
\hat{\mu}_{2}\\
\hat{\mu}_{3}\\
\hat{\mu}_{4}\\
\hat{\mu}_{5}\\
\hat{\mu}_{6} \end{array} \right) = \left( \begin{array}{l}
1.5\\
1.5\\
0\\
0\\
0.35\\
0 \end{array} \right).
\end{equation*}
Moreover, using the same performance constraint of $\mathcal{Q}_\text{M,allow} = 210 \times 10^6$ corresponds to an approximate allowable deviation of $10\%$ in the maximum von Mises stress. As described in Subsection \ref{section:hyperrectangleConstruct}, univariate root-finding with this performance constraint constructs the sampling domain for the plate with a hole with six design parameters, which corresponds to:
\begin{equation*}
\bm{\tau}_\text{max} = \left( \begin{array}{c}
\tau_{\text{max},1}\\
\tau_{\text{max},2}\\
\tau_{\text{max},3}\\
\tau_{\text{max},4}\\
\tau_{\text{max},5}\\
\tau_{\text{max},6} \end{array} \right) = \left( \begin{array}{l}
0.25\\
0.231\\
0.263\\
0.098\\
0.034\\
0.3 \end{array} \right).
\end{equation*}
For this example, we also consider a nonzero $\bm{\tau}_\text{min}$:
\begin{equation*}
\bm{\tau}_\text{min} = \left( \begin{array}{c}
\tau_{\text{min},1}\\
\tau_{\text{min},2}\\
\tau_{\text{min},3}\\
\tau_{\text{min},4}\\
\tau_{\text{min},5}\\
\tau_{\text{min},6} \end{array} \right) = \left( \begin{array}{l}
0.025\\
0.023\\
0.026\\
0.010\\
0.003\\
0.03 \end{array} \right).
\end{equation*}
From here, we are capable of constructing the surrogate model via separated representations over a set of Monte Carlo samples. For this purpose, we use 1500 samples and once again perform a survey over various polynomial degrees and ranks until a desired surrogate model fidelity is obtained. The results of this survey are shown in Table \ref{table:6DPlateWithHoleSR} where the surrogate models' accuracies are determined by comparison to an additional set of 500 samples not used in the construction of the surrogate model. As expected, the higher-dimensionality of this problem necessitates the use of larger polynomial degrees and separation rank for comparable accuracy to the plate with a hole with two design parameters. Regardless, as depicted in the results, the methodology is still capable of representing the true response surface with excellent precision with a relatively few number of samples as well as low polynomial degrees and rank.	
	
	\begin{table}[!t] 
	 \centering 
	 \caption{Plate with a hole with six design parameters: The maximum von Mises stress surrogate modeling errors constructed from $N = 1500$ samples and $N_c = 500$.}
	  \label{table:6DPlateWithHoleSR} 	 	 \begin{tabular}{cccccccc}
& & $r=1$ & $r=2$ & $r=3$ & $r=4$ & $r=5$ & $r=6$ \\ 
\hhline{========}
\multirow{7}{*}{$\| {\bf e }_{r,p} \|_\infty$} & degree 0 & 4.918e-1 & & & & &  \\ 
& degree 1  & 1.503e-1  & 8.576e-2  &  &   &   &   \\ 
& degree 2  & 1.370e-1  & 3.982e-2  & 2.069e-2  &   &   &   \\ 
& degree 3  & 1.420e-1  & 3.664e-2  & 1.853e-2  & 8.061e-3  &  &   \\ 
& degree 4  &  & 3.774e-2  & 1.886e-2  & 1.234e-2  & 1.023e-2  &   \\ 
& degree 5  &   &   & 1.907e-2  & 1.303e-2  & 1.129e-2  & 4.686e-3  \\ 
& degree 6  &  & &  & 1.374e-2  & 1.012e-2  & 5.192e-3  \\ 
& degree 7  &  &   &   &   & 1.055e-2  & 4.804e-3  \\ 
& degree 8  &   &   &   &   &   & 3.908e-3  \\ 
\hhline{========}
\multirow{7}{*}{$\| {\bf e }_{r,p} \|_M$} & degree 0  & 1.047e-1 & & & & & \\ 
& degree 1  & 2.788e-2  & 1.832e-2  &   &   &   &  \\ 
& degree 2  & 2.159e-2  & 7.559e-3  & 3.355e-3  &  &   &  \\ 
& degree 3  & 2.169e-2  & 7.563e-3  & 3.243e-3  & 1.416e-3  &   & \\ 
& degree 4  &   & 7.661e-3  & 3.278e-3  & 2.000e-3  & 1.127e-3  &  \\ 
& degree 5  &  &   & 3.297e-3  & 2.033e-3  & 1.114e-3  & 5.916e-4 \\ 
& degree 6  &  &  &   & 2.066e-3  & 1.167e-3  & 6.155e-4 \\ 
& degree 7  &   &   &  &   & 1.221e-3  & 6.481e-4 \\ 
& degree 8  &  &   &   &   &   & 6.330e-4 \\ 
\hhline{========}
	 	 \end{tabular}
	 \end{table}

\begin{table}[!t] 
\centering
\caption{Plate with a hole with six design parameters: Optimal tolerance values obtained using a rank 20, degree 4 separated representation constructed from 7500 samples of the maximum von Mises stress between the top and bottom of the plate with a hole. These values are treated as the ``exact'' optima. Bold numbers indicate values lying on the boundary of the tolerance bounding box.}
\label{table:6dExactTol}
\begin{tabular}{c|ccc}
& $\mathcal{F}_1(\bm{\tau})$ & $\mathcal{F}_{\bm{\mu}}(\bm{\tau})$ & $\mathcal{F}_{-1}(\bm{\tau})$\\
\hline
$\hat{\tau}_1$ & 0.073 & 0.050 & 0.044\\
$\hat{\tau}_2$ & 0.058 & 0.080 & 0.032\\
$\hat{\tau}_3$ & 0.088 & 0.080 & 0.035\\
$\hat{\tau}_4$ & {\bf 0.010} & {\bf 0.010} & 0.023\\
$\hat{\tau}_5$ & {\bf 0.003} & {\bf 0.003} & 0.013\\
$\hat{\tau}_6$ & 0.117 & {\bf 0.030} & 0.136\\
\end{tabular}
\end{table}	 
	 
To assess the algorithm's convergence behavior, we construct a high-fidelity separated representation from 7500 samples of rank 20 and degree 4 to approximate $\hat{\bm{\tau}}$. This separated representation gives $\|{\bf e}_{r,p} \|_M =$ 4.837e-5 and $\| {\bf e}_{r,p} \|_\infty =$ 3.649e-4. The high-fidelity approximations to the optimal tolerances, which are treated as the ``exact'' optima, are presented in Table \ref{table:6dExactTol}.

\begin{table}[t!] 
	 \centering 
	 \caption{Plate with a hole with six design parameters: Errors and convergence behavior of the tolerance allocation algorithm for the $1$-norm.}
	  \label{table:PlateHole6D_1} 	 	 \begin{tabular}{cc| rrrrr}
$r$ & $p$ & $n_{GA}$ & $n_{CG}$ & $\| \bm{\epsilon}_{r,p} \|_\infty$ & $\varphi_{1} \left(  \bm{\tau}_{r,p} \right)$ & $\gamma_\text{M} \left(  \bm{\tau}_{r,p} \right)$ \\
\hhline{=======}
 \multirow{3}{*}{3} & 3 & 85 & 60    & 2.977e-3    & 1.807e-2   & 2.328e-3   \\
 & 4 & 85 & 60    & 3.103e-3    & 1.991e-2   & 2.525e-3  \\
 & 5 & 85 & 60    & 3.098e-3    & 2.132e-2   & 2.673e-3\\
 \hline
 \multirow{3}{*}{4} & 4 & 86 & 61    & 3.391e-3    & 8.553e-3   & 1.068e-3  \\
 & 5 & 88 & 60    & 4.338e-3    & 5.455e-3   & 7.665e-4  \\
 & 6 & 87 & 61    & 4.291e-3    & 2.801e-3   & 4.764e-4 \\
  \hline
 \multirow{3}{*}{5} & 5 & 88 & 62    & 3.814e-3    & 3.252e-3   & 4.292e-4 \\
 & 6 & 87 & 62    & 3.738e-3    & 3.911e-3   & 5.049e-4  \\
 & 7 & 88 & 62    & 3.703e-3    & 2.895e-3   & 3.555e-4  \\
  \hline
 \multirow{3}{*}{6} & 6 & 88 & 62    & 4.162e-3    & 6.215e-4   & 6.912e-7   \\
 & 7 & 88 & 62    & 4.146e-3    & 9.188e-4   & 4.253e-5  \\
 & 8 & 88 & 62    & 3.931e-3    & 1.062e-3   & 1.677e-4 \\
 \hhline{=======}
	 	 \end{tabular}
	 \end{table}

	 	 \begin{table}[t!] 
	 \centering 
	 \caption{Plate with a hole with six design parameters: Errors and convergence behavior of the tolerance allocation algorithm for the $\bm{\mu}$-norm.}
	  \label{table:PlateHole6D_mu} 	 	 \begin{tabular}{cc| rrrrr}
$r$ & $p$ & $n_{GA}$ & $n_{CG}$ & $\| \bm{\epsilon}_{r,p} \|_\infty$ & $\varphi_{\bm{\mu}} \left(  \bm{\tau}_{r,p} \right)$ & $\gamma_\text{M} \left(  \bm{\tau}_{r,p} \right)$ \\
 \hhline{=======}
 \multirow{3}{*}{3} & 3 & 29 & 19    & 2.288e-3    & 1.525e-2   & 1.518e-3   \\
 & 4 & 29 & 19    & 2.794e-3    & 1.694e-2   & 1.627e-3  \\
 & 5 & 29 & 20    & 2.534e-3    & 2.018e-2   & 1.982e-3 \\
 \hline
 \multirow{3}{*}{4} & 4 & 29 & 19    & 9.338e-4    & 6.433e-3   & 6.957e-4 \\
 & 5 & 28 & 20    & 1.087e-3    & 2.824e-3   & 3.928e-4  \\
 & 6 & 28 & 19    & 7.918e-4    & 3.853e-4   & 1.429e-4 \\
  \hline
 \multirow{3}{*}{5} & 5 & 29 & 20    & 6.914e-4    & 4.599e-3   & 4.073e-4  \\
 & 6 & 29 & 19    & 4.488e-4    & 4.502e-3   & 4.077e-4   \\
 & 7 & 29 & 21    & 1.052e-3    & 4.078e-3   & 3.331e-4 \\
  \hline
 \multirow{3}{*}{6} & 6 & 28 & 20    & 7.188e-4    & 3.603e-4   & 7.333e-5   \\
 & 7 & 28 & 22    & 8.143e-4    & 2.831e-4   & 1.011e-4   \\
 & 8 & 29 & 17    & 2.477e-2    & 2.276e-3   & 9.222e-5  \\
 \hhline{=======}
	 	 \end{tabular}
	 \end{table}
	 
	  \begin{table}[t!] 
	 \centering 
	 \caption{Plate with a hole with six design parameters: Errors and convergence behavior of the tolerance allocation algorithm for the $-1$-norm.}
	  \label{table:PlateHole6D_-1}  	 	 \begin{tabular}{cc| rrrrr}
$r$ & $p$ & $n_{GA}$ & $n_{CG}$ & $\| \bm{\epsilon}_{r,p} \|_\infty$ & $\varphi_{-1} \left(  \bm{\tau}_{r,p} \right)$ & $\gamma_\text{M} \left(  \bm{\tau}_{r,p} \right)$ \\
 \hhline{=======}
 \multirow{3}{*}{3} & 3 & $>100$ & $>100$ & 5.614e-2    & 4.055e-2   & 3.727e-3 \\
 & 4 & $>100$ & $>100$ & 1.897e-3    & 3.227e-2   & 3.891e-3   \\
 & 5 & $>100$ & $>100$ & 5.219e-3    & 3.483e-2   & 4.085e-3  \\
 \hline
 \multirow{3}{*}{4} & 4 & $>100$ & $>100$ & 2.028e-3    & 1.368e-2   & 1.567e-3  \\
 & 5 & $>100$ & $>100$ & 2.091e-2    & 1.607e-2   & 1.358e-3  \\
 & 6 & $>100$ & $>100$ & 2.477e-2    & 1.711e-2   & 1.398e-3  \\
  \hline
 \multirow{3}{*}{5} & 5 & $>100$ & $>100$ & 5.134e-4    & 7.612e-3   & 8.839e-4  \\
 & 6 & $>100$ & $>100$ & 4.924e-4    & 8.294e-3   & 9.653e-4  \\
 & 7 & $>100$ & $>100$ & 2.142e-4    & 6.807e-3   & 7.946e-4  \\
  \hline
 \multirow{3}{*}{6} & 6 & $>100$ & $>100$ & 4.160e-3    & 4.955e-3   & 4.670e-4 \\
 & 7 & $>100$ & $>100$  & 1.620e-3    & 4.479e-3   & 4.777e-4   \\
 & 8 & $>100$ & $>100$ & 2.569e-4    & 4.664e-3   & 5.443e-4  \\
 \hhline{=======}
	 	 \end{tabular}
		 \vspace{10pt}
	 \end{table}

Tables \ref{table:PlateHole6D_1}, \ref{table:PlateHole6D_mu}, and \ref{table:PlateHole6D_-1} depict the effectiveness of our algorithm with respect to polynomial degree and rank of the surrogate model. In these tables, the error in the obtained tolerance, the relative error in the obtained objective functional, and the relative error in the obtained constraint functional are all reported, and the number of iterations ($n_{GA}$ and $n_{CG}$ for manifold gradient ascent and manifold conjugate gradient, respectively) until the increase in allocated tolerance size is within $10^{-6}$ is also reported.  Note that, as compared with the plate with a hole with two design parameters, we do not see monotonic convergence in the errors.  In fact, the error in the obtained tolerance appears to stall with increasing rank $r$ and polynomial degree $p$ for both the $1$-norm, the $\bm{\mu}$-norm, and the $-1$-norm.  However, there appears to be convergence, albeit slow convergence, in the obtained objective functional and the obtained constraint functional.  For $r = 6$ and $p = 7$, the relative error in the objective functional is less than 0.1\% for the $1$-norm, 0.03\% for the $\bm{\mu}$-norm, and 0.5\% for the $-1$-norm, and the relative error in the constraint functional is less than 0.005\% for the $1$-norm, 0.02\% for the $\bm{\mu}$-norm, and 0.05\% for the $-1$-norm.  Note the relative errors in the constraint functional are all much smaller than the prescribed allowable deviation of $10\%$ in the maximum von Mises stress.  It should finally be noted that the number of iterations required to converge the tolerance for each $r$ and $p$ is much higher than for the plate with a hole with two design parameters.  This is especially the case for the $-1$-norm.  However, the manifold conjugate gradient method does require less iterations than the manifold gradient ascent method.
	 

\subsection{L-Bracket with Seventeen Design Parameters}

\begin{figure}[b!]
	\centering
	\begin{subfigure}[c]{.48\textwidth}
		\centering
		\includegraphics[scale=1.24]{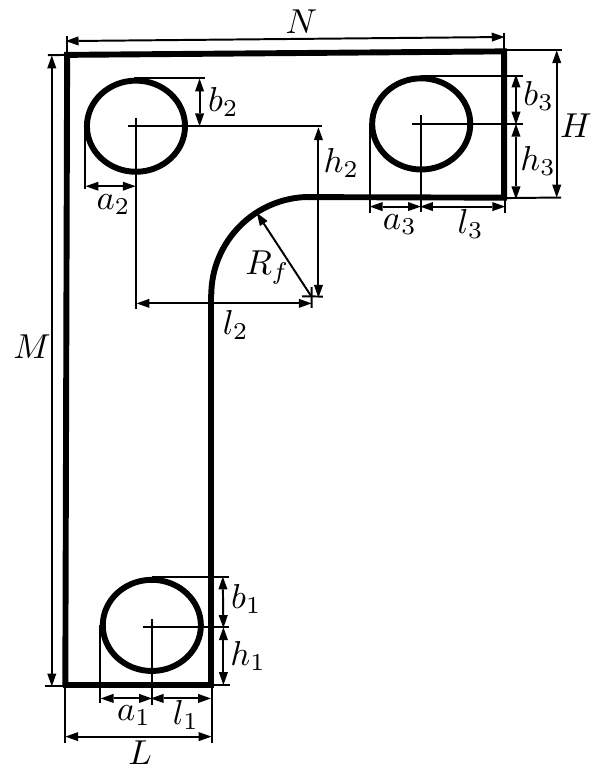}
	\end{subfigure}
	\begin{subfigure}[c]{.48\textwidth}
		\centering
		\includegraphics[scale=1.26]{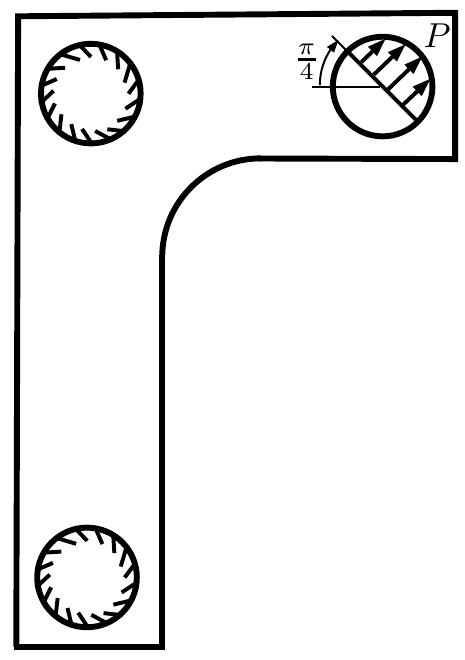}
	\end{subfigure}
	\caption{(left) L-Bracket geometric configuration. (right) The loading and boundary conditions associated with the L-Bracket problem. A uniform bearing pressure of $P=30 \times 10^6$ is applied to the top-right hole while the other two holes have zero displacement boundary conditions.}
	\label{fig:LBracket_config}
\end{figure}

\begin{figure}[t!]
	\centering
	\begin{subfigure}[c]{.5\textwidth}
		\centering
		\includegraphics[scale=1.25]{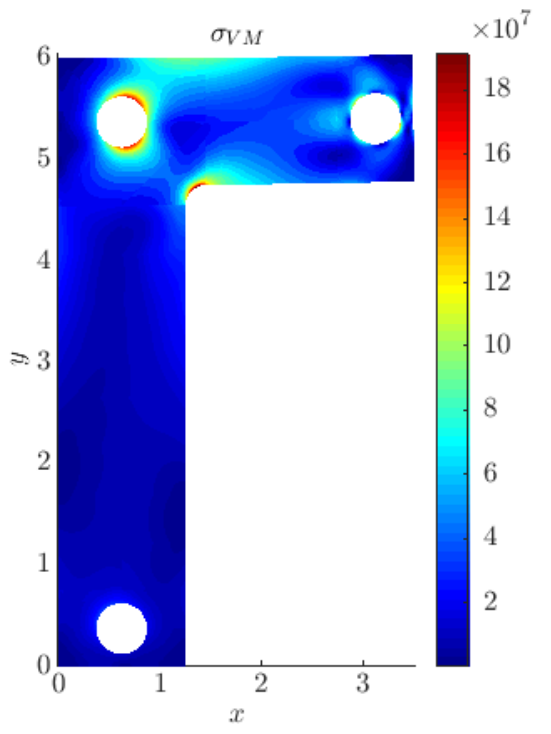}
	\end{subfigure}
	\caption{von Mises stress distribution for the L-Bracket nominal configuration.}
	\label{fig:LBracket_vonMises}
\end{figure}

The last example we consider is the structural deformation of an L-Bracket.  The geometric dimensions of the L-Bracket are displayed in the left of Fig. \ref{fig:LBracket_config}.  All seventeen geometric dimensions are expressed in terms of a design variable $\bm{\mu} \in \mathbb{R}^{17}$ via:
\begin{equation}
\begin{array}{llll}
M = \mu_1\mathscr{L}, & N = \mu_2\mathscr{L}, & L = \mu_3\mathscr{L}, & H = \mu_4\mathscr{L}, \\
R_f = \mu_5\mathscr{L}, & a_1 = \mu_6\mathscr{L}, & b_1 = a_1 \left( 1 - (\mu_7)^2 \right), & l_1 = \mu_8 \left( a_1 + \frac{L}{2} \right) + \frac{L}{2}, \\
h_1 = b_1 \left( 1 + \mu_9 \right), & a_2 = \mu_{10} \mathscr{L} , & b_2 = a_2 \left( 1 - ( \mu_{11})^2 \right), & l_2 = \left( R_f +\frac{L}{2} \right) \left( 1 + \mu_{12}\right),\\
h_2 = \left( R_f +\frac{H}{2} \right) \left( 1 + \mu_{13}\right), & a_3 = \mu_{14}\mathscr{L}, & b_3 = a_3 \left( 1 - (\mu_{15})^2 \right), & l_3 = a_3 \left( 1+\mu_{16}\right), \ \text{and} \\
 h_3 = \mu_{17} \left( b_3 + \frac{H}{2} \right) + \frac{H}{2}, & & & 
\end{array}
\end{equation}
where $\mathscr{L}=1$ is a chosen length scale. The nominal configuration for the L-Bracket is then:
\begin{align*}
\hat{\bm{\mu}} &= \left(
\hat{\mu}_{1},
\hat{\mu}_{2},
\hat{\mu}_{3},
\hat{\mu}_{4},
\hat{\mu}_{5},
\hat{\mu}_{6},
\hat{\mu}_{7},
\hat{\mu}_{8},
\hat{\mu}_{9},
\hat{\mu}_{10},
\hat{\mu}_{11},
\hat{\mu}_{12},
\hat{\mu}_{13},
\hat{\mu}_{14},
\hat{\mu}_{15},
\hat{\mu}_{16},
\hat{\mu}_{17} \right)^T \\
&= \left(
6,
3.5,
1.25,
1.25,
0.2,
0.25,
0,
0,
0,
0.25,
0,
0,
0,
0.25,
0,
0,
0 \right)^T.
\end{align*}
The L-Bracket is assumed to be made of an isotropic material with Young's modulus $E = 200 \times 10^9$ and Poisson ratio $\nu = 0.3$, and the L-Bracket is assumed to be in a plane stress state.  The loading and boundary conditions are depicted in the right of Fig. \ref{fig:LBracket_config}.  In particular, a uniform bearing pressure of $P=30 \times 10^6$ is applied to the top-right hole while the other two holes have zero displacement boundary conditions.  The other boundaries of the L-Bracket are subject to a zero traction boundary condition.  To obtain the structural deformation of the L-Bracket under the applied loading and boundary conditions, we discretize the L-Bracket with a 1792-element, multi-patch isogeometric analysis model parametrized with quadratic NURBS functions.  The von Mises stress distribution for the nominal configuration is displayed in Fig. \ref{fig:LBracket_vonMises}.

We use a performance constraint of $\mathcal{Q}_\text{M,allow} = 210 \times 10^6$ which corresponds to an approximate allowable deviation of $10\%$ in the maximum von Mises stress located at the top of the fillet. Univariate root-finding with this performance constraint constructs the sampling domain for this problem and sets the value of $\bm{\tau}_{\text{max}}$.  We also consider a nonzero $\bm{\tau}_{\text{min}}$.  The particular values of $\bm{\tau}_{\text{max}}$ and $\bm{\tau}_{\text{min}}$ are:
\begin{equation*}
\begin{array}{lr}
\bm{\tau}_{\text{max}} = \left( \begin{array}{c}
\tau_{\text{max},1} \\
\tau_{\text{max},2}   \\
\tau_{\text{max},3}  \\
\tau_{\text{max},4}  \\
\tau_{\text{max},5}  \\
\tau_{\text{max},6}   \\
\tau_{\text{max},7}   \\
\tau_{\text{max},8}   \\
\tau_{\text{max},9}   \\
\tau_{\text{max},10}  \\
\tau_{\text{max},11}   \\
\tau_{\text{max},12}   \\
\tau_{\text{max},13}  \\
\tau_{\text{max},14}   \\
\tau_{\text{max},15}  \\
\tau_{\text{max},16}  \\
\tau_{\text{max},17}   \end{array} \right) = \left( \begin{array}{l}
0.5 \\
0.208\\
0.25\\
0.071\\
0.038\\
0.1\\
0.3\\
0.6\\
0.6\\
0.037\\
0.3\\
0.566\\
0.474\\
0.025\\
0.3\\
0.6\\
0.6 \end{array} \right)   & \bm{\tau}_\text{min} = \left( \begin{array}{c}
\tau_{\text{min},1} \\
\tau_{\text{min},2}   \\
\tau_{\text{min},3}  \\
\tau_{\text{min},4}  \\
\tau_{\text{min},5}  \\
\tau_{\text{min},6}   \\
\tau_{\text{min},7}   \\
\tau_{\text{min},8}   \\
\tau_{\text{min},9}   \\
\tau_{\text{min},10}  \\
\tau_{\text{min},11}   \\
\tau_{\text{min},12}   \\
\tau_{\text{min},13}  \\
\tau_{\text{min},14}   \\
\tau_{\text{min},15}  \\
\tau_{\text{min},16}  \\
\tau_{\text{min},17}   \end{array} \right) = \left( \begin{array}{l}
0.05 \\
0.021\\
0.025\\
0.007\\
0.004\\
0.01\\
0.03\\
0.06\\
0.06\\
0.004\\
0.03\\
0.057\\
0.047\\
0.003\\
0.03\\
0.06\\
0.06 \end{array} \right)
\end{array}
\end{equation*}
From here, as before, we are capable of constructing the surrogate model via separated representations over a set of Monte Carlo samples.  We use 6000 samples for surrogate model construction and 500 additional samples to assess surrogate model accuracy.  In Table \ref{table:17dSR}, we report the surrogate modeling errors as a function of rank and polynomial degree.  From the table, we see that accuracy improves with increasing rank provided there is a corresponding increase in polynomial degree.  Despite the high-dimensionality of the design space, the low rank, separated representation surrogate models are able to achieve a high level of accuracy at relatively low rank.  For instance, for $r = 16$ and $p = 5$, the average relative error is less than 0.04\% and the maximum relative error is less than 0.2\% among the considered samples.

To assess the tolerance allocation algorithm's convergence behavior, we construct a high-fidelity separated representation from 7500 samples of rank 20 and degree 4 to approximate $\hat{\bm{\tau}}$. This separated representation gives $\|{\bf e} \|_M = 2.271$e-4 and $\| {\bf e} \|_\infty = 1.438$e-3. The high-fidelity approximations to the optimal tolerances, which are treated as the ``exact optima'', over the considered $\mathcal{G}(\bm{\tau})$ and $\mathcal{F}(\bm{\tau})$ are presented in Table \ref{table:17dExactTol}.

\begin{table}[!t] 
	 \centering 
	 \caption{L-Bracket with seventeen design parameters: The maximum von Mises stress surrogate modeling errors constructed from $N = 6000$ samples and $N_c = 500$.}
	 \label{table:17dSR} 	 	 \begin{tabular}{cccccccccc}
& & $r=9$ & $r=10$ & $r=11$ & $r=12$ & $r=13$ & $r=14$ & $r=15$ & $r=16$ \\ 
\hhline{==========} 
\multirow{5}{*}{$\| {\bf e }_{r,p} \|_\infty$}& degree 1  & 5.777e-2 &                    &                   &                    & 9.770e-2  &                  &                    &                    \\ 
& degree 2  & 5.730e-3 & 4.665e-3  &                    &                    & 6.340e-3  & 9.105e-3 &                     &                    \\ 
& degree 3  &                   & 2.087e-3  & 1.993e-3  &                    &                    & 7.317e-3 & 2.526e-3  &                    \\ 
& degree 4  &                   &                   &  2.773e-3 & 2.448e-3  &                     &                  & 2.052e-3 & 1.627e-3                 \\ 
& degree 5  &                   &                   &                   & 3.498e-3  &                     &                  &                             &   1.638e-3                 \\ 
\hhline{==========} 
\multirow{5}{*}{$\| {\bf e }_{r,p} \|_M$}& degree 1  & 1.266e-2 &                   &                     &                  & 1.337e-2 &                     &                    &                    \\ 
& degree 2  & 1.122e-3 & 1.048e-3 &                     &                   & 1.001e-3 &  9.657e-4 &                     &                    \\ 
& degree 3  &                  &  4.621e-4  & 4.236e-4  &                   &                   &  4.018e-4 &3.462e-4  &                    \\ 
& degree 4  &                   &                    & 5.170e-4  & 4.251e-4 &                     &                  & 3.518e-4 &  2.957e-4                 \\ 
& degree 5  &                   &                     &                   & 4.622e-4  &                    &                  &                  &   3.131e-4                 \\ 
\hhline{==========} 
	 	 \end{tabular}
	 \end{table}

\begin{table}[!t] 
\centering
\caption{L-Bracket with seventeen design parameters: Tolerance values obtained using a rank 20, degree 4 separated representation constructed from 7500 samples. These values are treated as the ``exact'' optima. Bold numbers indicate values lying on the boundary of the tolerance bounding box.}
\label{table:17dExactTol}
\begin{tabular}{c|ccc}
& $\mathcal{F}_{1}(\bm{\tau})$ & $\mathcal{F}_{\bm{\mu}}(\bm{\tau})$ & $\mathcal{F}_{-1}(\bm{\tau})$ \\
\hline
$\hat{\tau}_M$ & 0.348 & 0.056 & {\bf 0.050}\\
 $\hat{\tau}_N$ & {\bf 0.021} & {\bf 0.021} & 0.022\\
$\hat{\tau}_{L_1}$ & {\bf 0.025} & {\bf 0.036} & 0.028\\
$\hat{\tau}_{H_3}$ & {\bf 0.007} & {\bf 0.007} & 0.008\\
$\hat{\tau}_{R_f}$ & {\bf 0.004} & {\bf 0.004} & 0.006\\
$\hat{\tau}_{a_1}$ & {\bf 0.100} & {\bf 0.010} & 0.012\\
$\hat{\tau}_{e_1}$ & {\bf 0.030} & {\bf 0.030} & 0.033\\
$\hat{\tau}_{h_1}$ & {\bf 0.600} & {\bf 0.060} & 0.067\\
$\hat{\tau}_{k_1}$ & {\bf 0.600} & {\bf 0.060} & 0.067\\
$\hat{\tau}_{a_2}$ & {\bf 0.004} & {\bf 0.004} & 0.006\\
$\hat{\tau}_{e_2}$ & 0.071 & {\bf 0.030} & 0.034\\
$\hat{\tau}_{h_2}$ & {\bf 0.057} & 0.167 & 0.063\\
$\hat{\tau}_{k_2}$ & {\bf 0.047} & {\bf 0.047} & 0.053\\
$\hat{\tau}_{a_3}$ & {\bf 0.003} & {\bf 0.003} & 0.005\\
$\hat{\tau}_{e_3}$ & {\bf 0.300} & {\bf 0.030} & 0.034\\
$\hat{\tau}_{h_3}$ & {\bf 0.060} & 0.090 & {\bf 0.060}\\
$\hat{\tau}_{k_3}$ & 0.594 & {\bf 0.060} & 0.068
\end{tabular}
\end{table}

Tables \ref{table:LBracket_1}, \ref{table:LBracket_mu}, and \ref{table:LBracket_-1} depict the effectiveness of our algorithm with respect to polynomial degree and rank of the surrogate model for the L-Bracket with seventeen design parameters. In these tables, the error in the obtained tolerance, the relative error in the obtained objective functional, and the relative error in the obtained constraint functional are all reported, and the number of iterations $n_{GA}$ and $n_{CG}$ until the increase in allocated tolerance size is within $10^{-6}$ are also reported.  Note that, like the plate with a hole with six design parameters, the error in the obtained tolerance stalls, but there appears to be slow convergence in the obtained objective functional.  For $r = 16$ and $p = 5$, the relative error in the objective functional is less than 1\% for the $1$-norm, 0.4\% for the $\bm{\mu}$-norm, and 0.03\% for the $-1$-norm.  Unlike the plate with a hole with six design parameters, the error in the obtained constraint functional also stalls.  However, the error in the constraint functional is less than 0.5\% for all cases, considerably smaller than the prescribed allowable deviation of $10\%$ in the maximum von Mises stress.  Finally, it is noted that the manifold conjugate gradient method does require remarkably less iterations than the manifold gradient ascent method for this high-dimensional problem, especially for the $\bm{\mu}$-norm and the $-1$-norm.

 \begin{table}[t!] 
	 \centering 
	 \caption{L-Bracket with seventeen design parameters: Errors and convergence behavior of the tolerance allocation algorithm for the $1$-norm.}
	  \label{table:LBracket_1} 	 	 \begin{tabular}{cc| rrrrr}
r & p & $n_{GA}$ & $n_{CG}$ & $\| \bm{\epsilon}_{r,p} \|_\infty$ & $\varphi_{1} \left(  \bm{\tau}_\text{r,p} \right)$ & $\gamma_\text{M} \left(  \bm{\tau}_\text{r,p} \right)$ \\
 \hhline{=======}
\multirow{2}{*}{9} & 1 & 66 & 62    & 4.931e-1    & 2.690e-2   & 3.334e-3 \\
 & 2 & $>$100 & $>$100 & 2.700e-1    & 4.591e-2   & 1.468e-4   \\
\hline 
\multirow{2}{*}{10} & 2 & $>$100 & $>$100 & 2.700e-1    & 5.045e-2   & 2.244e-4  \\
 & 3 & $>$100 & 74 & 5.340e-1    & 1.094e-3   & 2.271e-4  \\
 \hline
 \multirow{2}{*}{11}  & 3& $>$100 & $>$100 & 5.340e-1    & 1.381e-1   & 2.021e-4  \\
 & 4 & $>$100 & 76 & 3.228e-1    & 1.112e-1   & 6.753e-5  \\
 \hline
 \multirow{2}{*}{12} & 4 & $>$100 & $>$100 & 2.043e-1    & 6.775e-2   & 1.370e-4  \\
 & 5 & $>$100 & 69 & 2.700e-1    & 3.035e-2   & 1.029e-4 \\
  \hline
 \multirow{2}{*}{13}  & 1 & 66 & 61    & 4.978e-1    & 2.670e-2   & 3.277e-3 \\
  & 2 & $>$100 & 42 & 2.953e-1    & 6.408e-2   & 2.531e-5 \\
  \hline
 \multirow{2}{*}{14}  & 2 & $>$100 & $>$100 & 3.279e-1    & 4.571e-2   & 9.149e-5 \\
 & 3 & $>$100 & $>$100 & 2.700e-1    & 8.513e-2   & 6.960e-5   \\
   \hline
 \multirow{2}{*}{15}  & 3 & $>$100 & 31    & 4.318e-1    & 1.013e-1   & 1.043e-4  \\
 & 4 & $>$100 & 38 & 3.754e-1    & 7.107e-2   & 1.457e-4 \\
   \hline
 \multirow{2}{*}{16} & 4 & $>$100 & 93 & 2.983e-1    & 3.304e-2   & 1.165e-4  \\
  & 5 & 70 & 24    & 4.972e-1    & 9.811e-3   & 2.646e-4 \\
 \hhline{=======}
	 	 \end{tabular}
	 \end{table}
	 
	 \begin{table}[t!] 
	 \centering 
	 \caption{L-Bracket with seventeen design parameters: Errors and convergence behavior of the tolerance allocation algorithm for the $\bm{\mu}$-norm.}
	  \label{table:LBracket_mu} 	 	 \begin{tabular}{cc| rrrrr}
r & p & $n_{GA}$ & $n_{CG}$ & $\| \bm{\epsilon}_{r,p} \|_\infty$ & $\varphi_{\bm{\mu}} \left(  \bm{\tau}_\text{r,p} \right)$ & $\gamma_\text{M} \left(  \bm{\tau}_\text{r,p} \right)$ \\
 \hhline{=======}
\multirow{2}{*}{9} & 1  & 49 & 11  & 4.695e-2    & 1.091e-1   & 2.984e-3  \\
 & 2 & $>$100 & 7 & 1.317e-1    & 1.459e-2   & 1.575e-4  \\
\hline 
\multirow{2}{*}{10} & 2 & 82 & 6    & 3.026e-2    & 9.584e-3   & 1.877e-4  \\
 & 3 & 54 & 7    & 3.799e-2    & 5.495e-4   & 1.985e-4 \\
 \hline
 \multirow{2}{*}{11}  & 3 & 62 & 7    & 4.082e-2    & 1.681e-2   & 1.966e-4\\
 & 4 & 56 & 7    & 3.876e-2    & 6.594e-3   & 4.562e-5   \\
 \hline
 \multirow{2}{*}{12} & 4 & 82 & 7    & 3.026e-2    & 1.394e-2   & 1.175e-4 \\
 & 5 & $>$100 & 7 & 4.103e-2    & 1.195e-2   & 1.910e-4 \\
  \hline
 \multirow{2}{*}{13}  & 1 & 49 & 11    & 4.705e-2    & 1.084e-1   & 2.929e-3  \\
  & 2 &  79 & 7   & 3.026e-2    & 1.457e-2   & 1.391e-5  \\
  \hline
 \multirow{2}{*}{14}  & 2 & 57  & 8  & 3.372e-2    & 4.805e-3   & 6.746e-5 \\
 & 3 &  75  & 8  & 3.430e-2    & 3.823e-2   & 5.397e-5   \\
   \hline
 \multirow{2}{*}{15}  & 3 & $>$100 & 7 & 1.144e-2    & 3.984e-3   & 3.428e-5  \\
 & 4 & 81  & 7  & 3.026e-2    & 9.893e-3   & 1.133e-4  \\
   \hline
 \multirow{2}{*}{16} & 4 &  62  & 7  & 7.662e-2    & 2.871e-3   & 7.815e-5  \\
  & 5 & $>$100 & 7 & 1.716e-2    & 3.503e-3   & 2.313e-4   \\
 \hhline{=======}
	 	 \end{tabular}
	 \end{table}
	 
	 \begin{table}[t!] 
	 \centering 
	 \caption{L-Bracket with seventeen design parameters: Errors and convergence behavior of the tolerance allocation algorithm for the $-1$-norm.}
	  \label{table:LBracket_-1} 	 	 \begin{tabular}{cc| rrrrr}
r & p & $n_{GA}$ & $n_{CG}$ & $\| \bm{\epsilon}_{r,p} \|_\infty$ & $\varphi_{-1} \left(  \bm{\tau}_\text{r,p} \right)$ & $\gamma_\text{M} \left(  \bm{\tau}_\text{r,p} \right)$ \\
 \hhline{=======}
\multirow{2}{*}{9} & 1  & $>$100 & $>$100  & 1.178e-2    & 1.214e-1   & 3.728e-3   \\
 & 2 &  65 & 13   & 8.653e-3    & 4.721e-3   & 1.048e-4   \\
\hline 
\multirow{2}{*}{10} & 2 & $>$100 & 10   & 6.578e-3    & 2.737e-3   & 1.458e-4   \\
 & 3 & $>$100 & 13 & 3.514e-2    & 5.525e-3   & 1.675e-4  \\
 \hline
 \multirow{2}{*}{11}  & 3& $>$100  & 4 & 8.075e-3    & 1.819e-3   & 1.534e-4  \\
 & 4 & 43  & 9  & 8.075e-3    & 1.816e-3   & 3.380e-6  \\
 \hline
 \multirow{2}{*}{12} & 4 & $>$100  & 8  & 8.075e-3    & 5.857e-3   & 6.364e-5 \\
 & 5 & $>$100 & 6 & 2.047e-2  & 3.132e-4   & 8.808e-5  \\
  \hline
 \multirow{2}{*}{13}  & 1 & $>$100 & $>$100 & 1.264e-2    & 1.207e-1   & 3.672e-3  \\
  & 2 & $>$100 & 7 & 2.000e-2    & 7.108e-4   & 7.152e-5   \\
  \hline
 \multirow{2}{*}{14}  & 2 & 25 & 16    & 7.454e-3    & 4.759e-3   & 5.788e-6 \\
 & 3 &  32 & 7   & 5.623e-3    & 1.305e-3   & 9.865e-6 \\
   \hline
 \multirow{2}{*}{15}  & 3 & 21 & 12    & 7.815e-3    & 3.915e-3   & 1.101e-5 \\
 & 4 & $>$100 & 8 & 3.190e-2    & 1.298e-3   & 4.391e-5 \\
   \hline
 \multirow{2}{*}{16} & 4 & $>$100 & 4 & 1.532e-2    & 2.906e-3   & 4.544e-6 \\
  & 5 & 40 & 7   & 5.924e-3    & 2.901e-4   & 1.157e-4  \\
 \hhline{=======}
	 	 \end{tabular}
	 \end{table}

\section{Conclusions}
\label{sec:conclusions}

In this paper, we have presented a novel tolerance allocation methodology which is suitable for geometric design configurations parameterized with moderate dimensionality. This approach naturally emanates from design space exploration techniques and parametric modeling paradigms. Although the methodology was presented in this paper for the setting of linear elasticity, it is overall agnostic with respect to the underlying physical model and performance constraints considered. Provided with a parametric PDE, a user is capable of allocating design tolerances based on prescribed performance constraints by solving an optimization problem over an immersed manifold of codimension one. We have presented both gradient ascent and conjugate gradient algorithms for performing optimization along this manifold to ultimately arrive at a tolerance for which all designs within the tolerance satisfy the prescribed performance constraint.  Moreover, to reduce computational cost, we proposed the use of low-rank, separated representation surrogate models to map the design parameter variation to the system performance. Numerical results presented, which included the plate with a hole described with two design parameters, the plate with a hole described with six design parameters, and the L-Bracket described with seventeen design parameters, demonstrate that this methodology is robust up to moderate dimensionality. However there is an incurred increase in computational expense due to the offline, separated representation construction, which requires a larger set of sample realizations to obtain a desired surrogate model fidelity.

This paper highlights several outstanding limitations which the authors plan to tackle in a future paper. First, the separated representation methodology used for surrogate model construction provides an excellent tool for efficient and accurate surrogate modeling with respect to \emph{smooth} system responses to smooth changes in design variables. However in many practical scenarios, the responses are not expected to be smooth, e.g. in the scenario where the \emph{location} of maximum stress changes in a discontinuous fashion with respect to a continuous change in design parameter. Utilizing a continuity-adaptive basis, rather than globally-smooth Legendre polynomials, for separated representation construction may remedy this issue and provide a means for attaining a tolerance allocation methodology which is capable of ensuring a conformity to geometrically-global, pointwise, worst-case performance criteria. Second, extending this framework to incorporate multiple constraint functionals would be a beneficial contribution in a variety of scenarios. For example, a physical system requiring conformity to a maximum stress and a maximum displacement can be ensured through a performance-based tolerance allocation routine of this form. Lastly, although worst-case tolerance allocation provides a measure of design conformity with respect to \emph{every} design within the prescribed tolerance, this is arguably too restrictive since, probabilistically speaking, the absolute worst-case scenario is extremely unlikely to occur in practice. To this end, we propose changing the constraint functional from a pointwise metric to one that is statistical, in particular, a metric which ensures that the designs contained within the prescribed tolerance hyperrectangle conform to the performance constraint in a statistically-average sense with respect to a provided probability density function.

\section{Acknowledgements}

AD acknowledges funding by the US Department of Energy's Office of Science Advanced Scientific Computing Research, Award DE-SC0006402, and National Science Foundation Grant CMMI-145460.  The authors would like to thank Matthew Reynolds for his advisement and expertise in separated representation construction and subsequent optimization protocols which contributed to the development of this work.

\section{References}

\renewcommand\refname{\vskip -1cm}
\nocite{*}

\bibliography{References.bib}{}

\begin{thebibliography}{10}

\bibitem{absil2007trust}
P.-A. Absil, C.~G. Baker, and K.~A. Gallivan.
\newblock {Trust-region methods on Riemannian manifolds}.
\newblock {\em Foundations of Computational Mathematics}, 7(3):303--330, 2007.

\bibitem{absil2009optimization}
P.-A. Absil, R.~Mahony, and R.~Sepulchre.
\newblock {\em Optimization Algorithms on Matrix Manifolds}.
\newblock Princeton University Press, 2009.

\bibitem{absil2012projection}
P.-A. Absil and J.~Malick.
\newblock Projection-like retractions on matrix manifolds.
\newblock {\em SIAM Journal on Optimization}, 22(1):135--158, 2012.

\bibitem{absil2015low}
P.-A. Absil and I.~V. Oseledets.
\newblock {Low-rank retractions: A survey and new results}.
\newblock {\em Computational Optimization and Applications}, 62(1):5--29, 2015.

\bibitem{battaglino2018practical}
C.~Battaglino, G.~Ballard, and T.~G. Kolda.
\newblock {A practical randomized CP tensor decomposition}.
\newblock {\em SIAM Journal on Matrix Analysis and Applications},
  39(2):876--901, 2018.

\bibitem{benzaken2017rapid}
J.~Benzaken, A.~J. Herrema, M.-C. Hsu, and J.~A. Evans.
\newblock A rapid and efficient isogeometric design space exploration framework
  with application to structural mechanics.
\newblock {\em Computer Methods in Applied Mechanics and Engineering},
  316:1215--1256, 2017.

\bibitem{beylkin2009multivariate}
G.~Beylkin, J.~Garcke, and M.~J. Mohlenkamp.
\newblock Multivariate regression and machine learning with sums of separable
  functions.
\newblock {\em SIAM Journal on Scientific Computing}, 31(3):1840--1857, 2009.

\bibitem{beylkin2002numerical}
G.~Beylkin and M.~J. Mohlenkamp.
\newblock Numerical operator calculus in higher dimensions.
\newblock {\em Proceedings of the National Academy of Sciences},
  99(16):10246--10251, 2002.

\bibitem{beylkin2005algorithms}
G.~Beylkin and M.~J. Mohlenkamp.
\newblock Algorithms for numerical analysis in high dimensions.
\newblock {\em SIAM Journal on Scientific Computing}, 26(6):2133--2159, 2005.

\bibitem{boothby1986introduction}
W.~M. Boothby.
\newblock {\em An Introduction to Differentiable Manifolds and Riemannian
  Geometry}, volume 120.
\newblock Academic Press, 1986.

\bibitem{borden2011isogeometric}
M.~J. Borden, M.~A. Scott, J.~A. Evans, and T.~J.~R. Hughes.
\newblock {Isogeometric finite element data structures based on B{\'e}zier
  extraction of NURBS}.
\newblock {\em International Journal for Numerical Methods in Engineering},
  87(1-5):15--47, 2011.

\bibitem{brent2013algorithms}
R.~P. Brent.
\newblock {\em Algorithms for Minimization Without Derivatives}.
\newblock Courier Corporation, 2013.

\bibitem{bro1997parafac}
R.~Bro.
\newblock {PARAFAC. Tutorial and applications}.
\newblock {\em Chemometrics and Intelligent Laboratory Systems},
  38(2):149--171, 1997.

\bibitem{chase1999tolerance}
K.~W. Chase.
\newblock Tolerance allocation methods for designers.
\newblock {\em ADCATS Report}, 99(6):1--28, 1999.

\bibitem{chase1988design}
K.~W. Chase and W.~H. Greenwood.
\newblock Design issues in mechanical tolerance analysis.
\newblock {\em Manufacturing Review}, 1(1):50--59, 1988.

\bibitem{chase1990least}
K.~W. Chase, W.~H. Greenwood, B.~G. Loosli, and L.~F. Hauglund.
\newblock Least cost tolerance allocation for mechanical assemblies with
  automated process selection.
\newblock {\em Manufacturing Review}, 3(1):49--59, 1990.

\bibitem{choi2000optimal}
H.-G.~R. Choi, M.-H. Park, and E.~Salisbury.
\newblock Optimal tolerance allocation with loss functions.
\newblock {\em Journal of Manufacturing Science and Engineering},
  122(3):529--535, 2000.

\bibitem{cottrell2009isogeometric}
J.~A. Cottrell, T.~J.~R. Hughes, and Y.~Bazilevs.
\newblock {\em Isogeometric Analysis: Toward Integration of CAD and FEA}.
\newblock John Wiley \& Sons, 2009.

\bibitem{de2000best}
L.~De~Lathauwer, B.~De~Moor, and J.~Vandewalle.
\newblock On the best rank-1 and rank-(r 1, r 2,..., rn) approximation of
  higher-order tensors.
\newblock {\em SIAM Journal on Matrix Analysis and Applications},
  21(4):1324--1342, 2000.

\bibitem{doostan2013non}
A.~Doostan, A.~A. Validi, and G.~Iaccarino.
\newblock Non-intrusive low-rank separated approximation of high-dimensional
  stochastic models.
\newblock {\em Computer Methods in Applied Mechanics and Engineering},
  263:42--55, 2013.

\bibitem{feng1997robust}
C.-X.~J. Feng and A.~Kusiak.
\newblock Robust tolerance design with the integer programming approach.
\newblock {\em Journal of Manufacturing Science and Engineering},
  119(4A):603--610, 1997.

\bibitem{flaig2002process}
J.~J. Flaig.
\newblock Process capability optimization.
\newblock {\em Quality Engineering}, 15(2):233--242, 2002.

\bibitem{fletcher1964function}
R.~Fletcher and C.~M. Reeves.
\newblock Function minimization by conjugate gradients.
\newblock {\em The Computer Journal}, 7(2):149--154, 1964.

\bibitem{gawlik2018high}
E.~S. Gawlik and M.~Leok.
\newblock High-order retractions on matrix manifolds using projected
  polynomials.
\newblock {\em SIAM Journal on Matrix Analysis and Applications},
  39(2):801--828, 2018.

\bibitem{heling2016connected}
B.~Heling, A.~Aschenbrenner, M.~S.~J. Walter, and S.~Wartzack.
\newblock On connected tolerances in statistical tolerance-cost-optimization of
  assemblies with interrelated dimension chains.
\newblock {\em Procedia CIRP}, 43:262--267, 2016.

\bibitem{herrema2017framework}
A.~J. Herrema, N.~M. Wiese, C.~N. Darling, B.~Ganapathysubramanian,
  A.~Krishnamurthy, and M.-C. Hsu.
\newblock A framework for parametric design optimization using isogeometric
  analysis.
\newblock {\em Computer Methods in Applied Mechanics and Engineering},
  316:944--965, 2017.

\bibitem{hsu2015interactive}
M.-C. Hsu, C.~Wang, A.~J. Herrema, D.~Schillinger, A.~Ghoshal, and Y.~Bazilevs.
\newblock An interactive geometry modeling and parametric design platform for
  isogeometric analysis.
\newblock {\em Computers \& Mathematics with Applications}, 70(7):1481--1500,
  2015.

\bibitem{huang2015broyden}
W.~Huang, K.~A. Gallivan, and P.-A. Absil.
\newblock {A Broyden class of quasi-Newton methods for Riemannian
  optimization}.
\newblock {\em SIAM Journal on Optimization}, 25(3):1660--1685, 2015.

\bibitem{HughesFEM}
T.~J.~R. Hughes.
\newblock {\em The Finite Element Method: Linear Static and Dynamic Finite
  Element Analysis}.
\newblock Courier Corporation, 2012.

\bibitem{hughes2005isogeometric}
T.~J.~R. Hughes, J.~A. Cottrell, and Y.~Bazilevs.
\newblock {Isogeometric analysis: CAD, finite elements, NURBS, exact geometry
  and mesh refinement}.
\newblock {\em Computer Methods in Applied Mechanics and Engineering},
  194(39-41):4135--4195, 2005.

\bibitem{huper2004newton}
K.~Huper and J.~Trumpf.
\newblock Newton-like methods for numerical optimization on manifolds.
\newblock In {\em Conference Record of the Thirty-Eighth Asilomar Conference on
  Signals, Systems and Computers, 2004.}, volume~1, pages 136--139. IEEE, 2004.

\bibitem{ji2000optimal}
S.~Ji, X.~Li, Y.~Ma, and H.~Cai.
\newblock Optimal tolerance allocation based on fuzzy comprehensive evaluation
  and genetic algorithm.
\newblock {\em The International Journal of Advanced Manufacturing Technology},
  16(7):461--468, 2000.

\bibitem{kroonenberg1980principal}
P.~M. Kroonenberg and J.~De~Leeuw.
\newblock Principal component analysis of three-mode data by means of
  alternating least squares algorithms.
\newblock {\em Psychometrika}, 45(1):69--97, 1980.

\bibitem{lin1997study}
C.-Y. Lin, W.-H. Huang, M.-C. Jeng, and J.-L. Doong.
\newblock {Study of an assembly tolerance allocation model based on Monte Carlo
  simulation}.
\newblock {\em Journal of Materials Processing Technology}, 70(1-3):9--16,
  1997.

\bibitem{ngoi1999optimum}
B.~K.~A. Ngoi and O.~J. Min.
\newblock Optimum tolerance allocation in assembly.
\newblock {\em The International Journal of Advanced Manufacturing Technology},
  15(9):660--665, 1999.

\bibitem{parkinson1993general}
A.~Parkinson, C.~Sorensen, and N.~Pourhassan.
\newblock A general approach for robust optimal design.
\newblock {\em Journal of Mechanical Design}, 115(1):74--80, 1993.

\bibitem{polak1969note}
E.~Polak and G.~Ribiere.
\newblock {Note sur la convergence de m{\'e}thodes de directions
  conjugu{\'e}es}.
\newblock {\em ESAIM: Mathematical Modelling and Numerical
  Analysis-Mod{\'e}lisation Math{\'e}matique et Analyse Num{\'e}rique},
  3(R1):35--43, 1969.

\bibitem{ramesh2009concurrent}
R.~Ramesh, J.~Jerald, T.~Page, and S.~Arunachalam.
\newblock Concurrent tolerance allocation using an artificial neural network
  and continuous ant colony optimisation.
\newblock {\em International Journal of Design Engineering}, 2(1):1--25, 2009.

\bibitem{reynolds2017optimization}
M.~J. Reynolds, G.~Beylkin, and A.~Doostan.
\newblock Optimization via separated representations and the canonical tensor
  decomposition.
\newblock {\em Journal of Computational Physics}, 348:220--230, 2017.

\bibitem{reynolds2016randomized}
M.~J. Reynolds, A.~Doostan, and G.~Beylkin.
\newblock {Randomized alternating least squares for canonical tensor
  decompositions: Application to a PDE with random data}.
\newblock {\em SIAM Journal on Scientific Computing}, 38(5):A2634--A2664, 2016.

\bibitem{scott2011isogeometric}
M.~A. Scott, M.~J. Borden, C.~V. Verhoosel, T.~W. Sederberg, and T.~J.~R.
  Hughes.
\newblock {Isogeometric finite element data structures based on B{\'e}zier
  extraction of T-splines}.
\newblock {\em International Journal for Numerical Methods in Engineering},
  88(2):126--156, 2011.

\bibitem{sibileau2018explicit}
A.~Sibileau, A.~Garc{\'\i}a-Gonz{\'a}lez, F.~Auricchio, S.~Morganti, and
  P.~D{\'\i}ez.
\newblock {Explicit parametric solutions of lattice structures with proper
  generalized decomposition (PGD)}.
\newblock {\em Computational Mechanics}, 62(4):871--891, 2018.

\bibitem{smith1994optimization}
S.~T. Smith.
\newblock {Optimization techniques on Riemannian manifolds}.
\newblock {\em Fields Institute Communications}, 3(3):113--135, 1994.

\bibitem{spotts1973allocation}
M.~F. Spotts.
\newblock Allocation of tolerances to minimize cost of assembly.
\newblock {\em Journal of Engineering for Industry}, 95(3):762--764, 1973.

\bibitem{wu2009improved}
F.~Wu, J.-Y. Dantan, A.~Etienne, A.~Siadat, and P.~Martin.
\newblock {Improved algorithm for tolerance allocation based on Monte Carlo
  simulation and discrete optimization}.
\newblock {\em Computers \& Industrial Engineering}, 56(4):1402--1413, 2009.

\bibitem{xiu2002wiener}
D.~Xiu and G.~E. Karniadakis.
\newblock {The Wiener--Askey polynomial chaos for stochastic differential
  equations}.
\newblock {\em SIAM Journal on Scientific Computing}, 24(2):619--644, 2002.

\bibitem{zhang1993integrated}
C.~Zhang and H.-P.~B. Wang.
\newblock Integrated tolerance optimisation with simulated annealing.
\newblock {\em The International Journal of Advanced Manufacturing Technology},
  8(3):167--174, 1993.

\end{thebibliography}
\bibliographystyle{plain}

\end{document}